\documentclass[a4paper,10pt]{scrartcl}
\usepackage[utf8]{inputenc}
\usepackage[english]{babel}
\usepackage[fixlanguage]{babelbib}
\usepackage{cite}
\usepackage{amssymb, amsmath, amstext, amsopn, amsthm, amscd, amsxtra, amsfonts}
\usepackage{yfonts, mathrsfs}
\usepackage{paralist}
\usepackage{colonequals}
\usepackage{enumerate}

\makeatletter
\def\namedlabel#1#2{\begingroup
	#2%
	\def\@currentlabel{#2}%
	\phantomsection\label{#1}\endgroup
}
\makeatother

\usepackage[all]{xy}

\usepackage[T1]{fontenc}
\usepackage{tikz}
\usetikzlibrary{shapes,arrows,positioning}

\usepackage{ifthen}
\usepackage{colonequals}
\usepackage{array}

\usepackage[bookmarks,bookmarksnumbered,plainpages=false]{hyperref}

\allowdisplaybreaks

\swapnumbers
\newtheoremstyle{standard}
 {16pt}  
 {16pt}  
 {}  
 {}  
 {\bfseries}
 {}  
 { } 
 {{\thmname{#1~}}{\thmnumber{#2.}}\thmnote{~(#3)}} 

\newtheoremstyle{kursiv}
 {16pt}  
 {16pt}  
 {\itshape}  
 {}  
 {\bfseries}
 {}  
 { } 
 {{\thmname{#1~}}{\thmnumber{#2.}}\thmnote{~(#3)}} 

\theoremstyle{standard}
\newtheorem{defn} [subsection]{Definition}
\newtheorem{ex} [subsection]{Example}
\newtheorem{rem}   [subsection]{Remark}
\newtheorem{nota}   [subsection]{Notation}
\newtheorem{setup} [subsection]{}

\theoremstyle{definition}

\theoremstyle{kursiv}
\newtheorem{thm}[subsection]{Theorem}
\newtheorem{prop} [subsection]{Proposition}

\newtheorem{lem} [subsection]{Lemma}

\setdefaultenum{\textup (a)}{\textup i.}{\textup A.}{\textup 1.}

\newcommand{\Evol}{\mathrm{Evol}}

\newcommand{\evol}{\mathrm{evol}}

\newcommand{\di}{\mathrm{d}}

\newcommand{\N}{\mathbb{N}}

\newcommand{\R}{\mathbb{R}}
\newcommand{\K}{\mathbb{K}}
\newcommand{\C}{\mathbb{C}}

\newcommand{\g}{\mathfrak{g}}

\renewcommand{\epsilon}{\varepsilon}

\newcommand{\set}[1]{\{  #1 \}}
\newcommand{\setm}[2]{\left\{\, #1 \middle\vert #2\,\right\}}
\newcommand{\norm}[1]{\left\lVert #1 \right\rVert}

\newcommand{\Bnorm}[1]{\norm{#1}_B}
\newcommand{\abs}[1]{\left| #1 \right|}
\newcommand{\opnorm}[1]{\norm{#1}_\text{op}}
\newcommand{\ve}{\varepsilon}
\newcommand{\coloneq}{\colonequals}
\DeclareMathOperator{\res}{res}
\DeclareMathOperator{\Hom}{Hom}

\DeclareMathOperator{\colours}{\mathcal{C}}

\DeclareMathOperator{\elcopro}{\overline{\Delta}^\ve}

\DeclareSymbolFont{bbold}{U}{bbold}{m}{n}
\DeclareSymbolFontAlphabet{\mathbbold}{bbold}
\newcommand{\Cut}[1]{\Lambda}

\newcommand{\cH}{\ensuremath{\mathcal{H}}}

\newcommand{\cP}{\ensuremath{\mathcal{P}}}

\newcommand{\cT}{\ensuremath{\mathcal{T}}}

\newcommand{\Lf}{\ensuremath{\mathbf{L}}}

\newcommand{\func}[5]{#1 \colon #2 \rightarrow #3 ,\quad #4 \mapsto #5}
\newcommand{\smfunc}[3]{#1 \colon #2 \rightarrow #3}
\newcommand{\nnfunc}[4]{#1 \rightarrow #2 ,\quad #3 \mapsto #4}
\newcommand{\smset}[1]{ \left\{ #1 \right\} }
\newcommand{\oBallin}[3]{U_{#1}^{#2}\left(#3 \right)}

\newcommand{\LB}[1][\cdot \hspace{1pt} , \cdot]{[\hspace{1pt} #1 \hspace{1pt} ]}

\DeclareMathOperator{\one}{\mathbf{1}}
\newcommand{\Frechet}{Fr\'echet }

\newcommand{\Hopf}  {\cH}					
\newcommand{\lcA}{A}						
\newcommand{\lcB}{B}						
\newcommand{\InfChar}[2]{ \g(#1 , #2) }				
\newcommand{\Char}[2]{     G(#1 , #2) }			
 \newcommand{\controlledChar}[2]{G_\mathrm{ctr}(#1,#2)}
 \newcommand{\controlledInfChar}[2]{\g_\mathrm{ctr}(#1,#2)}

\newcommand{\catname}[1]{\ensuremath{\normalfont{\mathbf{#1}}}}
\newcommand{\CoHopf}{\catname{CombHopf}}

\makeatletter
\providecommand*{\shuffle}{%
  \mathbin{\mathpalette\shuffle@{}}%
}
\newcommand*{\shuffle@}[2]{%
  \sbox0{$#1\vcenter{}$}%
  \kern .15\ht0 
  \rlap{\vrule height .25\ht0 depth 0pt width 2.5\ht0}%
  \raise.1\ht0\hbox to 2.5\ht0{%
    \vrule height 1.75\ht0 depth -.1\ht0 width .17\ht0 %
    \hfill
    \vrule height 1.75\ht0 depth -.1\ht0 width .17\ht0 %
    \hfill
    \vrule height 1.75\ht0 depth -.1\ht0 width .17\ht0 %
  }%
  \kern .15\ht0 
}
\makeatother

\newcommand{\RT}{\ensuremath{\cT}}
\DeclareMathOperator{\OST}{OST}

\newcommand{\anf}[1]{``#1''}

\renewcommand{\phi}{\varphi}

 \newcommand{\Gr}[1]{\omega_{#1}}

 \renewcommand{\projlim}[1]{\lim\limits_{\substack{\longleftarrow \\#1 }}}

 \newcommand{\Ind}{J}
 \newcommand{\IndwAbs}{(\Ind,\abs{\cdot})}

 \newcommand{\varInd}{I}
 \newcommand{\varIndwAbs}{(\varInd,\abs{\cdot})}

 \newcommand{\ellOneNorm}[2]{\norm{#1}_{\ell^1_{#2}}}
 \newcommand{\ellOne}[2]{\ell^1_{#1}(#2)}
 \newcommand{\ellOneLimit}[1]{\ellOne{\leftarrow}{#1}}

 \newcommand{\ellInfty}[3]{\ell^\infty_{#1}(#2,#3)}
 \newcommand{\ellInftyLimit}[2]{\ellInfty{\rightarrow}{#1}{#2}}

 \newcommand{\FreeMonoid}[1]{#1^*}
 \newcommand{\FreeCommMonoid}[1]{#1^\dagger}
 \newcommand{\emptyword}{1_M}

\tikzstyle dtree=[grow'=up,sibling distance=2mm,level distance=2mm,thick]
\tikzstyle dtree node=[scale=0.3,shape=circle,very thin,draw]
\tikzstyle dtree black node=[style=dtree node,fill=black]
\tikzstyle dtree white node=[style=dtree node,fill=white]

\title{Lie groups of controlled characters of combinatorial Hopf algebras} 
\author{R.\ Dahmen\footnote{Rafael Dahmen, Technische Universit\"{a}t Darmstadt, Germany }\ \ 
        and A.\ Schmeding\footnote{Alexander Schmeding, TU Berlin, Germany}} 
\begin{document}

\maketitle

\begin{abstract}
In this article groups of \anf{controlled} characters of a combinatorial Hopf algebra are considered from the perspective of infinite-dimensional Lie theory.
A character is controlled in our sense if it satisfies certain growth bounds, e.g.\ exponential growth. We study these characters for combinatorial Hopf algebras. Following Loday and Ronco, a combinatorial Hopf algebra is a graded and connected Hopf algebra which is a polynomial algebra with an explicit choice of basis (usually identified with combinatorial objects such as trees, graphs, etc.).
If the growth bounds and the Hopf algebra are compatible we prove that the controlled characters form infinite-dimensional Lie groups.
Further, we identify the Lie algebra and establish regularity results (in the sense of Milnor) for these Lie groups. The general construction principle exhibited here enables to treat a broad class of examples from physics, numerical analysis and control theory.

Groups of controlled characters appear in renormalisation of quantum field theories, numerical analysis and control theory in the guise of groups of locally convergent power series. 
The results presented here generalise the construction of the (tame) Butcher group, also known as the controlled character group of the Butcher-Connes-Kreimer Hopf algebra.
\end{abstract}

\medskip


\medskip

  (Secondary)

\tableofcontents

\section*{Introduction and statement of results} \addcontentsline{toc}{section}{Introduction and statement of results}
 Hopf algebras and their character groups appear in a variety of mathematical and physical contexts, such as numerical analysis \cite{MR2790315,MR2407032}, renormalisation of quantum field theories \cite{MR2523455,MR2371808}, the theory of rough paths \cite{MR3300969,1701.01152v1} and control theory \cite{MR3278760,Oberst}.
 In the contexts mentioned above, these Hopf algebras encode combinatorial information specific to the application. 
 On the structural level, the combinatorial nature of the Hopf algebras allows one to identify them with polynomial algebras
 built from combinatorial objects like trees, graphs, words or permutations. 
 Thus one is naturally led to the notion of a \emph{combinatorial Hopf algebra}. Following Loday and Ronco, a Hopf algebra is \emph{combinatorial} if it is a graded and connected Hopf algebra which is isomorphic as an algebra to a polynomial algebra.\footnote{Note that the isomorphism to the polynomial algebra is part of the structure of a combinatorial Hopf algebra. The notion of combinatorial Hopf algebra used in the present paper is more specialised than the notion of combinatorial Hopf algebras as defined in \cite{MR2732058}. In contrast to ibid.\ combinatorial Hopf algebras in our sense are always graded and connected and we restrict to Hopf algebras which are free algebras over a generating set (in ibid. also the cofree case is considered).}
 In this picture, the variables which generate the polynomial algebra correspond to the combinatorial objects.
 However, in applications one is usually not only interested in the Hopf algebra, but in its character group, i.e.\ the group of algebra homomorphisms from the Hopf algebra into a commutative target algebra.
 Elements in the character group can for example be identified with formal power series. These correspond in numerical analysis to formal numerical solutions of differential equation \cite{BS16} and to Chen-Fliess series in control theory \cite{DuffautEspinosa2016609}. Further, in the renormalisation of quantum field theories the correspondence yields (formal) power series acting on the space of coupling constants \cite[Chapter 1, 6.5]{MR2371808}.
  
 A Lie group structure for the character groups of a large class of Hopf algebras was constructed in \cite{BDS16}. 
 However, it turns out that the topology and differentiable structure of the character group is not fine enough to treat for example (local) convergence and well-posedness problems which arise in applications.
 As discussed in \cite{BS14,BDSOverview} this problem can not be solved in general, since the group of characters does not admit a finer structure which still turns it into a Lie group.
 
 One way to remedy this problem is to pass to a suitable subgroup of characters, which admits a finer structure.
 The groups envisaged here correspond to groups of power series which converge at least locally. 
 Groups arising in this way have been investigated in numerical analysis \cite{BS16} and in the context of control theory, where in \cite{MR3278760} a group of locally convergent Chen-Fliess series is constructed.
 For example, the group treated in numerical analysis, the so called tame Butcher group, corresponds to a group of formal power series, the B-series. 
 Local convergence of these power series yields locally convergent approximative (numerical) solutions to ordinary differential equations.
 Using the Lie group structure, it was observed in \cite{BS16} that methods from numerical analysis (e.g.\ identifying characters with B-series) correspond to Lie group morphisms between infinite-dimensional groups. Consequently, these mechanisms fit into an abstract framework where tools from Lie theory are available for the analysis.
 \smallskip
 
 Motivated by these results, we present a construction principle for Lie groups of ``controlled characters''. 
 Here a character is called controlled (a notion made precise in Section \ref{section_controlled_monoid} below), if it satisfies certain growth bounds, e.g.\ exponential growth bounds (tame Butcher group case).
 The general consensus in the literature seems to be that these groups should admit a suitable Lie group structure. 
 However, to obtain groups of controlled characters, the growth bounds and the Hopf algebra structure must be compatible.
 In essence this means that the structural maps of the Hopf algebra should satisfy certain estimates with respect to the growth bounds.
 This leads to the notion of a control pair (see Definition \ref{defn_control_pair} below) and we prove that every control pair gives rise to a Lie group of controlled characters, for which we investigate the Lie theoretic properties.
The novelty of these results is the flexibility of the construction which applies to a broad class of Hopf algebras and growth bounds (whereas the methods from the tame Butcher case were limited in scope). 
 
 Finally, we mention that in the present paper we mainly focus on the Lie theory for these groups and illustrate the constructions only with some selected examples.
 The reason for this is that examples usually require certain specific (and involved) combinatorial estimates.\footnote{We refer the reader to Section \ref{sect: examples} for a more in depth discussion of these problems.}
 It should come at no surprise that the general theory for controlled character groups is technically much more demanding then the general Lie theory for character groups as developed in \cite{BDS16}.
 The additional difficulty corresponds exactly to the passage from formal power series to (locally) convergent series. 
 However, the results obtained are stronger, as the topology and differentiable sturcture of controlled characters are suitable to treat applications where convergence is essential. 
 \smallskip
 
 We will now describe the results of this paper in more details. 
 Fix a graded and connected Hopf algebra $(\Hopf, m, 1_\Hopf, \Delta, \epsilon, S)$ and a commutative Banach algebra $B$ over $\K \in \{\R, \C\}$.
 We recall that he character group $\Char{\Hopf}{\lcB}$ of $\Hopf$ is defined as
  \begin{displaymath}
   \Char{\Hopf}{\lcB} \coloneq \left\{\phi \in \Hom_\K (\Hopf, \lcB) \middle| \substack{\phi (ab) = \phi (a)\phi (b), \forall a,b \in \Hopf\\ \text{ and } \phi (1_\Hopf) = 1_\lcB}\right\},
  \end{displaymath}
 together with the convolution product $\phi \star \psi \coloneq m_B \circ (\phi \otimes \psi ) \circ \Delta$.
 Our aim is to study certain subgroups of ``controlled characters'' for a combinatorial Hopf algebra  $(\Hopf, \Sigma)$.\footnote{As explained above, in our sense a \emph{combinatorial Hopf algebra} is a graded and connected Hopf algebra $\Hopf$ together with an isomorphism to an algebra of (non) commutative polynomials generated by $\Sigma \subseteq \Hopf$.}
 To give meaning to the term ``controlled'', we fix a family $(\Gr{k})_{k\in \N}$ of functions $\Gr{k} \colon \N_0 \rightarrow [0,\infty[$ which defines growth bounds for the characters.
 With respect to such a family, we then define the subset of all controlled characters as
  \begin{displaymath}
  \controlledChar{\Hopf}{B} = \setm{\phi \in \Char{\Hopf}{B} }{ \,\exists k \in \N \text{ with }
  \norm{\phi (\tau)}_B \leq \Gr{k} (\abs{\tau}) \text{ for all }\tau \in \Sigma}   .
  \end{displaymath}
 For example, the growth family $(n \mapsto 2^{kn})_{k \in \N}$ (used in \cite{BS16} for the tame Butcher group) gives rise to the subset of exponentially bounded characters.
 In general, the set $\controlledChar{\Hopf}{B}$ will not form  a subgroup of $(\Char{\Hopf}{\lcB}, \star)$ (cf.\ Example \ref{ex: counterex}). 
 However, if the growth family satisfies suitable conditions and is compatible with the combinatorial Hopf algebra, we prove that $\controlledChar{\Hopf}{B}$ is indeed a subgroup of the character group.
 Here compatibility roughly means that the coproduct and the antipode of the combinatorial Hopf algebra satisfy $\ell^1$-type estimates with respect to growth bounds.
 This leads to the notion of a control pair consisting of a combinatorial Hopf algebra $(\Hopf , \Sigma)$ and a compatible growth family (made precise in Section \ref{section_controlled_Hopf}). 
 For a control pair 
 $\left( 	(\Hopf , \Sigma), (\Gr{k})_{k\in \N}      \right)$
 and a Banach algebra $B$, we prove that the controlled characters form a group with a  natural (in general infinite-dimensional) Lie group structure:\medskip
 
 \textbf{Theorem A}\emph{
  Let $(\Hopf, \Sigma)$ be a combinatorial Hopf algebra, $B$ a commutative Banach algebra and $(\Gr{k})_{k\in \N}$ be a growth family such that $(\Hopf, (\Gr{k})_{k\in \N})$ forms a control pair. 
  The associated group of $B$-valued controlled characters $\controlledChar{\Hopf}{B}$ is an analytic Lie group modelled on an (LB)-space.}\medskip
 
To understand these results, we note that we use a concept of differentiable maps between locally convex spaces known as Bastiani's calculus \cite{MR0177277} or Keller's $C^r_c$-theory~\cite{keller1974} (see \cite{milnor1983,hg2002a,neeb2006} for streamlined expositions).
 For the reader's convenience Appendix \ref{App: lcvx:diff} contains a brief recollection of the basic notions used throughout the paper.
 At this point it is worthwhile to discuss a finer point of the topology and manifold structure of the controlled character group. 
 Recall from \cite{BDS16} that the character group $\Char{\Hopf}{B}$ of a graded and connected Hopf algebra is an infinite-dimensional Lie group. 
 Albeit the controlled characters form a subgroup of this Lie group, they are not a Lie subgroup in the sense that its manifold structure and the topology are \textbf{not} the induced ones.
 Instead, the topology on the group of controlled characters is properly finer than the one induced by all characters.\footnote{For additional information we refer to the discussion on the topology of the tame Butcher group in \cite{BS16} and in \cite{BDSOverview}. Note that the topologies there are studied with a view towards application in numerical analysis (compare \cite[Remark 2.3]{BS14}).} 
However, the inclusion $\controlledChar{\Hopf}{B} \rightarrow \Char{\Hopf}{B}$ turns out to be a smooth group homomorphism.

The Lie algebra of $\Char{\Hopf}{B}$ is the Lie algebra of infinitesimal characters
 \begin{displaymath}
  \InfChar{\Hopf}{\lcB} \coloneq \{ \phi \in \Hom_\K (\Hopf, \lcB) \mid \phi (ab) = \epsilon (a) \phi(b) + \epsilon(b)\phi(a)\}
 \end{displaymath}
 with the commutator Lie bracket $\LB[\phi ,\psi] \coloneq \phi \star \psi - \psi \star \phi$. 
 It turns out that the Lie algebra of the subgroup $\controlledChar{\Hopf}{B}$ consists precisely of those infinitesimal characters which satisfy an estimate similar to the definition of the controlled characters:
 \medskip
 
 \textbf{Lie algebra of the controlled characters}
 \emph{The Lie algebra of $\controlledChar{\Hopf}{B}$ is given by the controlled infinitesimal characters $(\controlledInfChar{\Hopf}{B},\LB )$, where the Lie bracket is the commutator bracket induced by the convolution}
 \begin{displaymath}
  \LB[\varphi , \psi] = \varphi \star \psi - \psi \star\varphi.
 \end{displaymath} 
 As in the Lie group case, this Lie algebra is a Lie subalgebra of the Lie algebra of all characters but in general with a finer topology. 
 
 Then we discuss regularity (in the sense of Milnor) for Lie groups of controlled characters.
 To understand these results first recall the notion of regularity for Lie groups.
  Let $G$ be a Lie group modelled on a locally convex space, with identity element $\one$, and
 $r\in \N_0\cup\{\infty\}$. We use the tangent map of the left translation
 $\lambda_g\colon G\to G$, $x\mapsto gx$ by $g\in G$ to define
 $g.v\coloneq T_{\one} \lambda_g(v) \in T_g G$ for $v\in T_{\one} (G) =: \Lf(G)$.
 Following \cite{HG15reg}, $G$ is called
 \emph{$C^r$-semiregular} if for each $C^r$-curve
 $\eta\colon [0,1]\rightarrow \Lf(G)$ the problem
 \begin{displaymath}
  \begin{cases}
   \gamma'(t)&= \gamma(t).\eta(t)\\ \gamma(0) &= \one
  \end{cases}
 \end{displaymath}
 has a (necessarily unique) $C^{r+1}$-solution
 $\Evol (\eta)\coloneq\gamma\colon [0,1]\rightarrow G$.
 Furthermore, if $G$ is $C^r$-semiregular and the map
 \begin{displaymath}
  \evol \colon C^r([0,1],\Lf(G))\rightarrow G,\quad \gamma\mapsto \Evol
  (\gamma)(1)
 \end{displaymath}
 is smooth, then $G$ is called \emph{$C^r$-regular}.
 If $G$ is $C^r$-regular and $r\leq s$, then $G$ is also
 $C^s$-regular. A $C^\infty$-regular Lie group $G$ is called \emph{regular}
 \emph{(in the sense of Milnor}) -- a property first defined in \cite{milnor1983}.
 Every finite-dimensional Lie group is $C^0$-regular (cf. \cite{neeb2006}). Several
 important results in infinite-dimensional Lie theory are only available for
 regular Lie groups (see
 \cite{milnor1983,neeb2006,HG15reg}, cf.\ also \cite{KM97}).
 \smallskip
 
 Recall that in \cite[Theorem C]{BDS16} the regularity of the character group $\Char{\Hopf}{B}$ is shown. 
 Together with well known techniques from \cite{HG15reg}, this result paves the way to establish semiregularity of the controlled character groups.
 Our strategy is to prove that the solutions of the differential equations on $\Char{\Hopf}{B}$ remain in the group of controlled characters if the initial data of the differential equation is contained in the controlled characters.
 Unfortunately, we were unable to obtain the necessary estimates to carry out this program in full generality. 
 Instead our methods yield (semi-)regularity only for a certain class of combinatorial Hopf algebras in the presence of additional estimates.
 Before we state the result, recall the notion of a \emph{right-handed Hopf algebra} from \cite{MP15}: 
 A combinatorial Hopf algebra $(\Hopf, \Sigma)$ is \emph{right-handed} if the reduced coproduct\footnote{Recall that in a graded and connected Hopf algebra the coproduct can be written as $\Delta (x) = 1 \otimes x + x\otimes 1 + \overline{\Delta} (x)$, where $\overline{\Delta}$ is a sum of tensor products $a\otimes b$ with $|a|, |b|\geq 1$.} satisfies
 \begin{displaymath}
  \overline{\Delta} (\tau) \subseteq \K^{(\Sigma)} \otimes \Hopf \quad \forall \tau \in \Sigma,
 \end{displaymath}
 where $\K^{(\Sigma)}$ is the vector space spanned by $\Sigma$. Thus in a right-handed Hopf algebra, the reduced coproduct contains only polynomials in $\Sigma$ on the right-hand side of the tensor products, whereas the left-hand side contains only (multiples of) elements of $\Sigma$.
 Many combinatorial Hopf algebras appearing in applications (e.g.\ the Butcher-Connes-Kreimer algebra, Fa\`{a} di Bruno algebras) are right-handed Hopf algebras.
 In addition to the right-handed condition we also need to assume that the terms of the reduced coproduct whose right-hand side is also given by a multiple of an element in $\Sigma$ grows only linearly in the degree.\footnote{Our methods can be adapted to slightly more general situations, cf.\ Remark \ref{rem: lineargrowth}.}
 To this end, we consider the elementary coproduct, i.e.\ the terms of the reduced coproduct contained in $\K^{(\Sigma)} \otimes \K^{(\Sigma)}$, 
 \begin{displaymath}
  \elcopro (\tau) := \sum_{\substack{|\alpha| + |\beta| = |\tau|\\ \alpha, \beta \in \Sigma}} c_{\alpha,\beta,\tau} \alpha \otimes \tau \quad \text{ for } \tau \in \Sigma.
 \end{displaymath}
 A right-handed Hopf algebra \emph{right-hand linearly bounded (RLB)} if there exist $a,b \in \R$ 
 \begin{displaymath}
  \text{such that }\norm{\elcopro (\tau)}_{\ell^1} = \sum_{\substack{|\alpha| + |\beta| = |\tau|\\ \alpha, \beta \in \Sigma}} |c_{\alpha,\beta,\tau}| \leq a|\tau|+b \quad \text{for all } \tau \in \Sigma.
 \end{displaymath}
 For example, the Butcher-Connes-Kreimer algebra of rooted trees is an (RLB) Hopf algebra.
 This property was used to establish the regularity of the tame Butcher group in \cite{BS16}.
 For groups of controlled characters of an (RLB)-Hopf algebra we prove.\medskip
 
 \textbf{Theorem B} \emph{Let $(\Hopf, \Sigma)$ be an (RLB) Hopf algebra, $B$ a Banach algebra. Assume that $((\Hopf, \Sigma), (\Gr{k})_{k\in \N})$ is a control pair.
 Then the associated group $\controlledChar{\Hopf}{B}$ is $C^0$-regular and thus in particular regular in the sense of Milnor.} \medskip
 
 The main difficulty here is to prove that the solutions of the differential equations stay bounded, whence the above estimates come into play.  In particular, Theorem B emphasises the value of ``right-handedness'' of the Hopf algebra\footnote{Right-handed Hopf algebras are closely connected to preLie algebras, cf.\ \cite{MR2732058}. In particular, it has been shown that they admit a Zimmerman type forrest formula, see \cite{MP15}.}.  
 Note that Theorem B is new even for the tame Butcher group, which was previously only known to be $C^1$-regular (see \cite{BS16}). 
 This seemingly minor improvement is quite important as \cite{Hanusch18} recently established Lie theoretic tools such as the strong Trotter formula for $C^0$-regular infinite-dimensional Lie groups. It is currently unknown whether similar results can be obtained for $C^1$-regular Lie groups. 
 We remark that the proof of Theorem B is much more conceptual than the proof for the regularity appearing in \cite{BS16}.

 Though Theorem B establishes regularity for controlled character groups for some combinatorial Hopf algebras, a general regularity result is still missing.  
 This is somewhat problematic as there are diverse combinatorial Hopf algebras which appear in applications and are not right-handed.
 For example, the shuffle algebra (cf.\ Example \ref{ex: shuffle}) is not right-handed so it is not a (RLB) Hopf algebra. 
 Also the Fa\`{a} di Bruno algebra, see Example \ref{ex_faadiBruno}, is not an (RLB) Hopf algebra (though it is right-handed)
 However, to the authors knowledge Theorem B is the only available result which asserts regularity of groups of controlled characters. 
 Stronger results could conceivably be established once new and refined methods become available.  
 It is a long standing conjecture in infinite-dimensional Lie theory, that every Lie group modelled on a complete space is regular.
 Hence we expect the groups to be regular and pose:\smallskip
 
 \textbf{Open Problem} \emph{Establish regularity for the group of controlled characters of an arbitrary combinatorial Hopf algebra. Alternatively, give an example of a combinatorial Hopf algebra with a a controlled character group which is not regular.}
  
 \paragraph{Acknowledgements} 
 This research was partially supported by the project \emph{Structure Preserving Integrators, Discrete Integrable Systems and Algebraic Combinatorics} (Norwegian Research Council project 231632) and the European Unions Horizon 2020 research and innovation
programme under the Marie Sk\l{}odowska-Curie grant agreement No.\ 691070.
 The authors thank A.\ Murua, J.M.\ Sanz-Serna, D.\ Manchon, K. Ebrahimi-Fard, W.S.\ Gray and S.\ Paycha for helpful discussions. Finally, we thank the anonymous referee for insightful comments which helped improve the article.


 \section{Foundations: Weighted function spaces}
 \begin{nota}
   We write $\N\coloneq \smset{1,2,3,\ldots}$, and $\N_0 \coloneq\N\cup\smset{0}$. 
   Throughout $\K \in \{\R,\C\}$. In a normed space $E$, we denote by $\oBallin{r}{E}{x}$ the $r-$ball around $x \in E$.
 \end{nota}
 
 Our goal is to study functions which grow in a controlled way, i.e.~ satisfy certain growth restrictions. 
 To make this precise we fix a family of functions which will measure how fast a function can grow while still be regarded as \anf{controlled}:
 \begin{defn}[Growth family]							\label{defn_growth_family}
  A family $\left(	\Gr{k}	\right)_{k\in\N}$ of functions $\smfunc{\Gr{k}}{\N_0}{\N}$ is called \emph{growth family} if it satisfies the following conditions. For all $k \in \N$ and $n,m\in \N_0$
  \begin{enumerate}
   \item[\namedlabel{axiom_H_monotonic_in_k}{(W1)}] $\Gr{k}(0)=1$  and $\Gr{k}(n)\leq \Gr{k+1}(n)$.
    \item[\namedlabel{axiom_H_Multi}{(W2)}] 
				      $
				      \Gr{k}(n)\cdot \Gr{k}(m) \leq  \Gr{k}(n+m). 
				      $
  \end{enumerate}
  \begin{enumerate}
   \item[\namedlabel{axiom_H_Infinity}{(W3)}] 
  						For all $k_1 \in \N$ there exists $k_2\geq k_1$ such that
				      $
				       \Gr{k_2}(n) \geq 2^n \cdot \Gr{k_1}(n)		.
				      $
  \end{enumerate}
  A growth family is called \emph{convex} if for a fixed $k_1$ we can choose $k_2\geq k_1$ such that in addition for each $k_3\geq k_2$ 
  \begin{enumerate}
   \item[\namedlabel{axiom_H_convexity}{(cW)}] $\exists \alpha \in ]0,1[ \text{ such that }\quad \Gr{k_1}(n)^\alpha\cdot \Gr{k_3}(n)^{1-\alpha}\leq\Gr{k_2}(n)	\quad \forall n\in\N_0$.
  \end{enumerate}
 \end{defn}
 
 The technical condition \ref{axiom_H_convexity} assures that the model spaces of the Lie groups of controlled characters are well behaved (see Lemma \ref{lem_creg_lim} for the exact statement).
 In general, condition \ref{axiom_H_convexity} is needed for our techniques. However, in an important class of examples (i.e.\ characters of Hopf algebras of finite type with values in finite-dimensional algebras, compare Proposition \ref{prop_properties_of_ellInftyLimit} (d)) it is not necessary to require \ref{axiom_H_convexity}.
  
 \begin{rem}
       Observe that \ref{axiom_H_Infinity} implies that a growth family will grow at least exponentially fast. 
       The growth bounds considered in Proposition \ref{prop: growth:fam} below thus grow at least with exponential speed.
       In general, it is desirable to use the most restrictive growth family (i.e.\ one with exponential growth) to ensure certain convergence properties in applications. 
        However, as we will see in Example \ref{ex_fdB2}, not all combinatorial Hopf algebras admit groups of controlled characters which are exponentially bounded.
 \end{rem}

 \begin{prop}[Standard examples] \label{prop: growth:fam}
  The following are convex growth families in the sense of Definition \ref{defn_growth_family} (Hopf algebra and Example refer to further information given in Section \ref{sect: examples}).\\
  \begin{minipage}{13cm}\renewcommand{\arraystretch}{1.5}
  \begin{center}\emph{ \begin{tabular}{|c|c|c|c|c|} \hline
    $\Gr{k}(n) $ & Application & Hopf algebra  & see &  Source \\ \hline  
	$k^n$     & B-series, & (coloured) Connes-Kreimer&\ref{ex: CKHopf}&\cite{HL1997}\\ 
	$2^{kn}$ &  P-series & algebra of rooted trees &  \ref{ex: colCKHopf} &\cite{BS16}\footnote{Note that the growth families $k^n$ and $2^{kn}$ both realise groups of exponentially bounded characters.}  \\\hline
	$k^n \cdot n!$  & Chen-Fliess series & Fa\`{a} di Bruno type & \ref{ex: fdB:Chen} &\cite{MR2849486}  \\\hline
	$k^{n^2} $ & formal power & Fa\`{a} di Bruno  & \ref{ex_faadiBruno}  & \\
	$k^n (n!)^k$ & series\footnote{The growth bounds will be used to illustrate certain technical details of the theory.} &  & \ref{ex_fdB2}  &  \\
	\hline
   \end{tabular}}
   \end{center} 
   \end{minipage}
 \end{prop}
 \begin{proof}
  Property \ref{axiom_H_monotonic_in_k} can be established by trivial calculations for all examples. 
 
 The functions $\Gr{k}(n):=k^n$ are multiplicative, i.e.\ $\Gr{k}(n)\Gr{k}(m)=\Gr{k}(n+m)$ so \ref{axiom_H_Multi} holds. 
  Furthermore, setting $k_2:=2k_1$, \ref{axiom_H_Infinity} is satisfied.
  This proves that the family $\left(n\mapsto k^n\right)_{k\in \N}$ is a growth family. 

  To show that it is convex, i.e.~that it satisfies \ref{axiom_H_convexity}, we proceed as follows:
  Let $k_1\in\N$ be given and let as above $k_2=2k_1$. Now for fixed $k_3\geq k_2$ we use that 
  \begin{displaymath}
   \lim_{\alpha\to 1}k_1^\alpha\cdot k_3^{1-\alpha} = k_1<k_2,
  \end{displaymath}
  whence there is $\alpha\in\left]0,1\right[$ such that $k_1^\alpha\cdot k_3^{1-\alpha} < k_2$. For each $n\in\N$ we now have
  \begin{equation}\label{eq: alphaineq}
   \Gr{k_1}(n)^\alpha\cdot \Gr{k_3}(n)^{1-\alpha}=k_1^{n\alpha}\cdot k_3^{n(1-\alpha)} < k_2^n=\Gr{k_2}(n)
  \end{equation}
  which is what we wanted to show. 
  As $\Gr{k} (n) = 2^{kn}$ is a cofinal subsequence\footnote{Recall that subsequence $(a_{n_k})_{k}$ of $(b_n)_{n}$ is cofinal if for every $n \in \N$ there is $k \in \N$ with $n_k > n$.} of the family $(n \mapsto k^n)_{k\in \N}$, we immediately see that $(n \mapsto 2^{kn})_{k\in \N}$ is also a convex growth family.  
  Similarly, one establishes that $\Gr{k}(n) =k^{n^2}$ forms a convex growth family.
  
  Now, we turn to $\Gr{k}(n) = k^n n!$:
  It is well known that $1 \leq \binom{i+j}{j}=\frac{(i+j)!}{i!j!} \leq 2^{i+j}$ holds for all $i,j \in \N$. 
  Hence \ref{axiom_H_Multi} holds and if we choose $k_2 \geq 2k_1$ we see that \ref{axiom_H_Infinity} is satisfied.
  Multiplying \eqref{eq: alphaineq} with $n!$, \ref{axiom_H_convexity} is immediately satisfied.
  
  We now turn to the family $\Gr{k}(n)=k^n(n!)^k$ and compute for $k \in \N$ as follows 
    \begin{align*}
     \Gr{k}(i)\Gr{k}(j) &= k^i (i!)^{k} k^j (j!)^{k} = k^{i+j} (i!j!)^{k} \leq k^{i+j} ((i+j)!)^{k} = \Gr{k} (i+j).
    \end{align*}
  Thus \ref{axiom_H_Multi} is satisfied. Again choosing $k_2 = 2k_1$ allows one to see that \ref{axiom_H_Infinity} holds.
  Therefore, the functions $\Gr{k}(n) = k^n(n!)^k$ form a growth family. 
  To see that this family is convex, consider $k_3 \geq k_2 = 2k_1$. Now since $k_2>k_1$ we can choose $\alpha \in ]0,1[$ such that \eqref{eq: alphaineq} is satisfied (with respect to $k_3$).
  Enlarging $\alpha \in ]0,1[$ if necessary, we may also assume that $\alpha k_1 + (1-\alpha)k_3 < k_2$.
  With this choice of $\alpha$ we obtain for each $n\in \N$
  \begin{displaymath}
    \Gr{k_1}(n)^\alpha\Gr{k_3}(n)^{1-\alpha} = (k_1^\alpha k_3^{1-\alpha})^n (n!)^{\alpha k_1 + (1-\alpha)k_3} < k_2^n (n!)^{k_2} = \Gr{k_2}(n). \qedhere
  \end{displaymath}
 \end{proof}

  \begin{rem}
   Let $\smfunc{\phi}{\N_0}{\N_0}$ be a function with the following properties:
   \begin{itemize}
    \item $\phi(0)=0$ and $\phi(n)\geq n$
    \item $\phi(m)+\phi(n)\leq \phi(m+n)$
   \end{itemize}
  Then $\left( n\mapsto \Gr{k}(\phi(n)) \right)_{k\in\N}$ is a (convex) growth family provided that $\left(	\Gr{k}	\right)_{k\in\N}$ is a (convex) growth family. 
   In particular, this works for monomial maps $\phi(n)=n^j$ for a fixed $j\in\N$.
   This generalises the example $(n\mapsto k^{n^2})_{k\in\N}$ from Proposition \ref{prop: growth:fam}.
  \end{rem}

  \begin{rem}\label{rem: nongrowthfamily}
  The properties \ref{axiom_H_monotonic_in_k} and \ref{axiom_H_Multi} are needed to construct locally convex algebras on certain inductive limits of spaces of bounded functions. 
  Further, \ref{axiom_H_Infinity} and \ref{axiom_H_convexity} ensure that these inductive limits have certain desirable properties.
  We can not dispense with property \ref{axiom_H_Infinity} if we want to obtain groups of controlled characters (whereas \ref{axiom_H_convexity} is not crucial in this regard). 
  As an example consider the family 
  \begin{displaymath}
   \Gr{k}(n) := e^{-\frac{n}{k}} , \ k \in \N.
  \end{displaymath}
   This family arises from growth bounds for Fourier coefficients in numerical analysis (see \cite{MR3320934}).
   It turns out that it is not a growth family as it satisfies all properties (including convexity) except for \ref{axiom_H_Infinity}. 
   We return later to this in Example \ref{ex: counterex} and see that the characters whose growth is bounded in this sense do not even form a group.   
  \end{rem}

 \begin{defn}[Graded index sets]								\label{defn_graded_index_set}
  A \emph{graded index set} $\IndwAbs$ is a set $\Ind$, together with a \emph{grading}, i.e.~a function $\smfunc{\abs{\cdot}}{\Ind}{\N_0}$.
  For each $n\in\N_0$, we use the notation
  \[
   \Ind_n:=\{\tau\in\Ind , \abs{\tau}=n\}.
  \]
  We call $\IndwAbs$ \emph{pure} if $\Ind_0=\emptyset$,
  \emph{connected} if $\Ind_0$ has exactly one element,
  and \emph{of finite type} if every $\Ind_n$ is finite.
 \end{defn}

 \begin{setup}											\label{setup_graded_vector_space}
  Given a graded index set $\IndwAbs$, define the graded vector space
  \[
   \K^{(\Ind)}:=\setm{\sum_{\tau\in\Ind} c_\tau \cdot \tau}{\text{all but finitely many $c_\tau$ are zero}} =\bigoplus_{n\in\N_0}\K^{(\Ind_n)}
  \]
  of formal linear combinations of elements of $\Ind$. Every graded vector can be obtained in this fashion (using a non-canonical choice of a basis of each step). 
  We will always think of $\Ind$ as a subset (the standard basis) of $\K^{(\Ind)}$.
 \end{setup}
 
 Controlled characters will be modelled on weighted sequence spaces (see e.g.\ \cite{MR0226355}).
 
 \begin{setup}[Weighted $\ell^1$-spaces induced by growth families]
 Let $\IndwAbs$ be a graded index set and $\left(\Gr{k}	\right)_{k\in\N}$ a growth family.
  We interpret the elements of the growth family as weights and 
  define a weighted norm on $\K^{(\Ind)}$ as follows
  \[
   \ellOneNorm{\sum_{\tau\in\Ind} c_\tau \cdot \tau}{k} := \sum_{\tau\in\Ind} \abs{c_\tau} \Gr{k}(|\tau|).
  \]
  The completion with respect to this norm is the Banach space
  \[
   \ellOne{k}{\Ind}:=\setm{ \sum_{\tau\in\Ind} c_\tau \cdot \tau , \ c_\tau \in \K }{ \sum_{\tau\in\Ind} \abs{c_\tau} \Gr{k}(|\tau|) <\infty}
  \]
  which is isometrically isomorphic to $\ell^1(\Ind)$ (cf.\ \cite[Section 7]{MR1483073}) via
  \[
   \nnfunc{\ellOne{k}{\Ind}}{\ellOne{}{\Ind}}{\sum_\tau c_\tau \cdot \tau}{\sum_\tau c_\tau \Gr{k}(\abs{\tau}) \cdot \tau.}
  \]
  For a given growth family, $\ellOne{k+1}{\Ind}$ is a subspace of $\ellOne{k}{\Ind}$ (see Lemma \ref{lem_canonical_maps}).
  Hence every graded index set with a growth family gives rise to a sequence of Banach spaces 
    \begin{displaymath}
     \ellOne{1}{\Ind} \supseteq \ellOne{2}{\Ind} \supseteq \cdots \supseteq \ellOne{k}{\Ind} \supseteq \ellOne{k+1}{\Ind} \cdots.
    \end{displaymath} 
  The inclusions in the above sequence are continuous linear with operator norm at most $1$ (see again Lemma \ref{lem_canonical_maps}).
  Thus one can consider the (locally convex) projective limit 
  \[
   \ellOneLimit{\Ind}:=\bigcap_{k\in\N}\ellOne{k}{\Ind} = \lim_{\substack{\longleftarrow \\k }} \ellOne{k}{\Ind} = \setm{\sum_\tau c_\tau \cdot \tau}{\text{for each }k\in\N : \sum_\tau \abs{c_\tau} \Gr{k}(\abs{\tau})<\infty}.
  \]
  This space is a \Frechet space. 
  And if $\IndwAbs$ is of finite type it is even a \Frechet-Schwartz space\footnote{Iteratively constructing indices which satisfy \ref{axiom_H_Infinity} 
										    yields a cofinal sequence with compact inclusion maps $I_{k_1}^{k_2}$, $I_{k_2}^{k_{3}}, \ldots$, 
										    whence the limit is a \Frechet-Schwartz space.}
 (see \cite[p.\ 303]{MR1483073} for more information).
 \end{setup}

 In the following, we will often have to consider the continuity of maps $\ellOneLimit{\Ind} \rightarrow \ellOneLimit{\varInd}$ which arise as extensions of linear mappings $\smfunc{T}{\K^{(\Ind)}}{\K^{(\varInd)}}$.
 One can prove the following criterion for the continuity of the extension (the proof is in Appendix \ref{app: aux:func}).
 \begin{lem}							\label{lem_continuity_of_linear_maps}
  Let $\IndwAbs$ and $\varIndwAbs$ be two graded index sets.
  Then a linear map $\smfunc{T}{\K^{(\Ind)}}{\K^{(\varInd)}}$ extends to a (unique) continuous operator $\widetilde T \colon \ellOneLimit{\Ind} \rightarrow \ellOneLimit{\varInd}$, if and only if for each $k_1\in\N$ there is a $k_2\in\N$ and a $C>0$ such that 
  \[
   \ellOneNorm{T \tau}{k_1}\leq C \Gr{k_2}(\abs{\tau}) \quad \text{for all } \tau\in\Ind.
  \]
 \end{lem}

 \subsection*{Passage to spaces of bounded functions}
 In this section we construct the model spaces for the controlled character groups. 
 These arise as weighted $\ell^\infty$ spaces with values in a Banach space.
 Hence throughout this section, we choose and fix a Banach space $B$ (over $\K$) and let $\IndwAbs$ be a graded index set.
 Further, the spaces of controlled functions encountered will always be constructed with respect to some choice of growth family $(\Gr{k})_{k \in \N}$.
 In this setting we can define controlled functions and Banach spaces of controlled functions.
 
  \begin{setup}[Controlled functions]							
   \begin{enumerate}[(a)]
    \item 
      Let $k\in\N$. 
      A function $\smfunc{f}{\Ind}{B}$ is called \emph{$k$-controlled} if $\sup_{\tau} \frac{\Bnorm{f(\tau)}}{\Gr{k}(\abs{\tau})} <\infty$.
      The space of all $k$-controlled functions
      \[
       \ellInfty{k}{\Ind}{B}:=\setm{\smfunc{f}{\Ind}{B}}{ \sup_{\tau} \frac{\Bnorm{f(\tau)}}{\Gr{k}(\abs{\tau})} <\infty }
      \]
      with the norm $\norm{f}_{\ell^\infty_k}:=\sup_{\tau} \frac{\Bnorm{f(\tau)}}{\Gr{k}(\abs{\tau})}$ is a Banach space.
      In fact, the mapping 
       \[
	\nnfunc{\ellInfty{k}{\Ind}{B}}{\ellInfty{}{\Ind}{B}= \left\{f \colon \Ind \rightarrow B \mid \sup_{\tau }\Bnorm{f(\tau)} <\infty\right\}}{f}{\frac{f}{\Gr{k}(\abs{\cdot})}.}
      \]
      is an isometric isomorphism to the Banach space $\ellInfty{}{\Ind}{B}$. 
      Since growth families satisfy \ref{axiom_H_monotonic_in_k}, it is easy to see that $\ellInfty{k}{\Ind}{B} \subseteq \ellInfty{k+1}{\Ind}{B}$ and the inclusion map 
      \[
	\nnfunc{\ellInfty{k}{\Ind}{B}}{\ellInfty{k+1}{\Ind}{B}}{f}{f}
      \]
      is a continuous operator with operator norm at most $1$ for all $k \in \N$.
    \item 
      A function $\smfunc{f}{\Ind}{B}$ is called \emph{controlled} if there is a $k\in\N$ such that $f$ is $k$-controlled, i.e.~$f\in\ellInfty{k}{\Ind}{B}$.
      The \emph{space of all controlled functions} is denoted by $\ellInftyLimit{\Ind}{B}$. 
      This space is locally convex as it is the locally convex direct limit 
      \begin{align*}
	\ellInftyLimit{\Ind}{B} &\coloneq \setm{\smfunc{f}{\Ind}{B}}{\text{there is a $k\in\N$ such that }\sup_{\tau} \frac{\Bnorm{f(\tau)}}{\Gr{k}(\abs{\tau})} <\infty}  \\ 
			    &\hphantom{:}=\lim_{\substack{\longrightarrow\\ k}} \ellInfty{k}{\Ind}{B}
      \end{align*}
      where the bonding maps of the limit are the continuous inclusions from (a).
  \end{enumerate}\label{setup_controlled_maps}
\end{setup}
The spaces of controlled functions will allow us to build model spaces for the Lie groups which are constructed later.
We will now discuss properties of spaces of controlled functions.
Beginning with the $k$-controlled functions we will establish that the space of all $k$-controlled functions is closely connected to certain linear maps.
This will explain in which sense characters of Hopf algebras are $k$-controlled. 

\begin{defn}\label{defn: lin:kcontrolled} 
 Let $k \in \N$ and endow $\K^{(\Ind)}$ with the $\ell^1_k$ norm. A continuous linear map $\smfunc{\phi}{\K^{(\Ind)}}{B}$ is called \emph{$k$-controlled}.
 Note that every $k$-controlled linear map extends (uniquely) to a continuous linear map $\ellOne{k}{\Ind} \rightarrow B$ on the completion, which has the same operator norm.     
\end{defn}

\begin{setup}										\label{setup_linear_maps_defined_on_a_basis}
   As $\Ind$ is a basis of the vector space $\K^{(\Ind)}$, a linear map $\smfunc{\phi}{\K^{(\Ind)}}{B}$ is uniquely determined by its values on $\Ind$.
   This yields an isomorphism of vector spaces
   \[
    \nnfunc{\Hom(\K^{(\Ind)},B)}{B^\Ind}{\phi}{f_\phi:=\phi|_\Ind}
   \]
   with inverse $\nnfunc{B^\Ind}{\Hom(\K^{(\Ind)},B)}{f}{\phi_f}$,
   where $\smfunc{\phi_f}{\K^{(\Ind)}}{B}$ is the unique linear map extending $\smfunc{f}{\Ind}{B}$. 
   It is not hard to see that this isomorphism induces an isometric isomorphism $\{f \in \textup{Hom}_{\K} (\K^{(\Ind)},B) \mid f \text{ is } k-$controlled$\} \rightarrow \ellInfty{k}{\Ind}{B}$.
   For the readers convenience a proof of this fact is recorded in Lemma \ref{lem: k:isom:iso}
  \end{setup}

   With the direct limit topology $\ellInftyLimit{\Ind}{B}$ is an (LB)-space, i.e.~a direct limit of an injective sequence of Banach spaces. 
      In general, (LB)-spaces may not be well-behaved, e.g.\ they may fail to be Hausdorff.
      However, one can prove the following.

 \begin{prop}[Properties of $\ellInftyLimit{\Ind}{B}$]								\label{prop_properties_of_ellInftyLimit}
  Let $\IndwAbs$ be a graded index set, $\left(\Gr{k}	\right)_{k\in\N}$ a growth family and $B$ be a Banach space. 
  Then the following is satisfied.
  \begin{enumerate}[\upshape (a)]
   \item The linear map $\nnfunc{\ellInftyLimit{\Ind}{B}}{B^\Ind}{f}{f}$ is continuous (where $B^\Ind$ is endowed with the topology of pointwise convergence). Thus $\ellInftyLimit{\Ind}{B}$ is a Hausdorff space.
   \item The linear map
		$
		 \ellInftyLimit{\Ind}{B} \rightarrow \Hom\left(	\ellOneLimit{\Ind} , B	\right), f \mapsto \left(\sum_\tau c_\tau \cdot \tau	\mapsto \sum_\tau c_\tau f(\tau) \right)
		$
	       is bijective and continuous (where $\Hom\left(	\ellOneLimit{\Ind} , B	\right)$ is endowed with the topology of uniform convergence on bounded subsets (the so called strong topology)).
   \item Assume that $\IndwAbs$ is pure. Then for every $r>0$ and every controlled function $f$ there is a $k\in\N$ such that $\norm{f}_{\ell^\infty_k}<r$, i.e.		$
		 \ellInftyLimit{\Ind}{B} = \bigcup_{k\in\N} \oBallin{r}{\ellInfty{k}{\Ind}{B}}{f}.
		$
   \item The space $\ellInftyLimit{\Ind}{B}$ is complete, if one of the following statements hold
		\begin{itemize}
		  \item The growth family is convex
		  \item $\IndwAbs$ is of finite type and $B$ is finite-dimensional.
		\end{itemize}
  \end{enumerate}
 \end{prop}
 It is not clear to the authors up to this point whether the linear map in part (b) may fail to be a topological isomorphism for some choices of 
 $\left(	\Gr{k}	\right)_{k\in\N}$, $\IndwAbs$, and $B$. 
 However, we will not need that fact at all.
 \begin{proof}
  \begin{enumerate}[(a)]
 \item Follows directly from the universal property of locally convex direct limits applied to the family $\ellInfty{k}{\Ind}{B} \rightarrow B^\Ind, f\mapsto f$. 	     
       Since $B^\Ind$ is Hausdorff, continuity of the inclusion implies that $\ellInftyLimit{\Ind}{B}$ is Hausdorff as well. 
 \item Continuity of the linear map follows directly of the universal property of locally convex direct limits applied to 
 $
  \ellInfty{k}{\Ind}{B} \cong \Hom(\ellOne{k}{\Ind},B)\rightarrow \Hom(\ellOneLimit{\Ind},B).
 $
 Clearly the map is injective and it is surjective by Lemma \ref{lem_factorisation}.
 \item Consider $f\in\ellInftyLimit{\Ind}{B}$ and choose $k_1\in\N$ such that $f\in\ellInfty{k_1}{\Ind}{B}$.
 By property \ref{axiom_H_Infinity} there is a $k_2\geq k_1$ such that
 \[
  \norm{f}_{\ell^\infty_{k_2}} = 					\sup_{\tau\in\Ind}\frac{\Bnorm{f(\tau)}}{\Gr{k_2}(\abs{\tau})}
			\stackrel{\ref{axiom_H_Infinity}}{\leq} 	\sup_{\tau\in\Ind}\frac{\Bnorm{f(\tau)}}{2^{\abs{\tau}}\cdot  \Gr{k_1}(\abs{\tau})}
			\leq 	 					\sup_{\tau\in\Ind}\frac{\Bnorm{f(\tau)}}{2 \cdot  \Gr{k_1}(\abs{\tau})}
			= 						\frac{1}{2}\cdot \norm{f}_{\ell^\infty_{k_1}}.
 \]
 In the second inequality we used the fact that $\abs{\tau}\geq1$ since $\IndwAbs$ is pure by assumption.

 Repeating this process, we get a sequence $(k_m)_{m\in\N}$ such that $\left(\norm{f}_{\ell^\infty_{k_m}}\right)_{m\in\N}$ converges to zero.
 In particular, there is $k_m$ such that $\norm{f}_{\ell^\infty_{k_m}}$ is less than $r$.\qedhere
 \item By \cite[Corollary 6.5]{MR1977923}, completeness of an (LB)-space follows if we can show that the limit $\ellInftyLimit{\Ind}{B}$ is compactly regular, i.e.~for every compact set $K\subseteq \ellInftyLimit{\Ind}{B}$ there is a $k\in\N$ such that $K$ is a compact subset of $\ellInfty{k}{\Ind}{B}$.
 If the growth family is convex, compact regularity follows from Lemma \ref{lem_creg_lim}. 
 If on the other hand, $\IndwAbs$ is of finite type and the Banach space $B$ is finite-dimensional, we obtain compact regularity from Lemma \ref{lem_silva}.
 \end{enumerate}
 \end{proof}

 \section{Analytic manifolds of controlled characters on free monoids}													\label{section_controlled_monoid}
 Throughout this section, let $B$ be a fixed commutative\footnote{Most of the results will hold for noncommutative algebras as well -- but in Section \ref{section_controlled_Hopf} we will need commutativity anyway.}
 Banach algebra, i.e.~ a commutative associative unital $\K$-algebra with a complete submultiplicative norm $\Bnorm{\cdot}$ such that $\Bnorm{1_B}=1$. 
 Further, we fix a (convex) growth family $(\omega_k)_{k \in \N}$ with respect to which all weighted sequence spaces, e.g.\ $\ellInfty{k}{\Ind}{B}$, in this section will be constructed.
 
 \begin{defn}[Graded monoid]									\label{defn_graded_monoid}
  A \emph{graded monoid} $(M,\abs{\cdot})$ is a monoid $M$, together with a monoid homomorphism
  \[
   \smfunc{\abs{\cdot}}{M}{(\N_0,+).}
  \]
  In particular, every graded monoid is a graded index set in the sense of Definition \ref{defn_graded_index_set}.

  For each graded monoid $(M,\abs{\cdot})$ the vector space $\K^{(M)}$ of finite formal linear combinations introduced in \ref{setup_graded_vector_space} becomes a graded associative unital algebra by extending the monoid multiplication bilinearly.
 \end{defn}

 \begin{ex}\label{ex:monoids}
  Let $(\Sigma,\abs{\cdot})$ be a graded index set. There are two natural ways of constructing a graded monoid over $\K \in \{\R,\C\}$ from $(\Sigma,\abs{\cdot})$.
  \begin{itemize}
   \item [(a)] The \emph{free monoid} $\FreeMonoid{\Sigma}$ (also called the \emph{word monoid}) over the alphabet $\Sigma$ consists of all finite words. 
               The degree of a word $w\in \FreeMonoid{\Sigma}$ is the sum of the degrees of the letters of $w$. The neutral element of $\FreeMonoid{\Sigma}$ is the empty word (whose degree is zero).
	       The unital algebra $\K\langle \Sigma \rangle := \K^{(\FreeMonoid{\Sigma})}$ is the noncommutative polynomial algebra over the alphabet $\Sigma$, which is \emph{free algebra} (or \emph{tensor algebra}) over $\K^{(\Sigma)}$.
   \item [(b)] Identify two words in $\FreeMonoid{\Sigma}$ when they agree up to a permutation of the letters to obtain the \emph{free commutative monoid} $\FreeCommMonoid{\Sigma}$.
	       This is a graded monoid with respect to the same grading. 
	       The unital algebra $\K[\Sigma]:=\K^{(\FreeCommMonoid{\Sigma})}$ is the commutative polynomial algebra over the alphabet $\Sigma$ (the symmetric algebra over $\K^{(\Sigma)}$).
  \end{itemize}
  In both examples, the grading on the monoid is connected if and only if the grading on $\Sigma$ is pure and the monoid is of finite type if and only if $(\Sigma,\abs{\cdot})$ is of finite type.
 \end{ex}

 \begin{setup}
  By a $B$-valued \emph{character} of the graded monoid $(M,\abs{\cdot})$ we mean a monoid homomorphism $\smfunc{\chi}{M}{(B,\cdot)}$.
  The isomorphism of \ref{setup_linear_maps_defined_on_a_basis} maps the characters of $M$ bijectively to the characters of the algebra $\K^{(M)}$, i.e.~the algebra homomorphisms of $\K^{(M)}$ to the unital algebra $B$.
 \end{setup}

 \begin{defn}[Controlled characters]
  Let $M$ be a graded monoid.
  A character $\smfunc{\chi}{M}{B}$ is called a \emph{controlled character} of the monoid $M$ if it is a controlled function, i.e.~$\chi\in\ellInftyLimit{M}{B}$.
  The set of controlled characters endowed with the subspace topology of $\ellInftyLimit{M}{B}$ is denoted by $\controlledChar{M}{B}.$
 \end{defn}
 The main result will be the fact that for the free monoid $M$, this subset is very well-behaved, in particular it is an analytic submanifold (see Proposition \ref{prop_characters_form_submanifold}).
 \begin{lem}												\label{lem_controlled_functions_on_generators}
  Let $\smfunc{\chi}{M}{B}$ be a $B$-valued character. 
  Assume that $M$ is (as a monoid) generated by a subset $\Sigma\subseteq M$ and let $k\in\N$.
  Then the following are equivalent:
  \begin{enumerate}[\upshape (i)]
   \item $\chi\in\ellInfty{k}{M}{B}$ with $\norm{\chi}_{\ell^\infty_k}= 1$
   \item $\chi|_\Sigma\in\ellInfty{k}{\Sigma}{B}$ with $\norm{\chi|_\Sigma}_{\ell^\infty_k}\leq 1$
  \end{enumerate}
 \end{lem}
 \begin{proof}
  It is obvious that (i) implies (ii). To show the converse, assume that $\chi|_\Sigma$ is $k$-controlled with norm at most $1$.
  This means that for each $\tau\in\Sigma$ we have:
  \[
   \Bnorm{\chi(\tau)}\leq \Gr{k}(\abs{\tau}). 
  \]
  We will now show that $\smfunc{\chi}{M}{B}$ is also $k$-controlled of norm at most $1$.
  Let $w=\tau_1\cdots \tau_\ell\in M$ be a finite word of length $\ell\geq1$. Then we can estimate:
  \begin{align*}
   \Bnorm{\chi(w)}   &	= 	\Bnorm{\chi(\tau_1\cdots \tau_\ell)}
				= 	\Bnorm{\chi(\tau_1)\cdots \chi(\tau_\ell)}
				\leq 	\Bnorm{\chi(\tau_1)}\cdots\Bnorm{\chi(\tau_\ell)}
		      \\&	\leq	\Gr{k}(\abs{\tau_1})\cdots \Gr{k}(\abs{\tau_\ell}) \stackrel{\ref{axiom_H_Multi}}{\leq}	\Gr{k}(\abs{\tau_1}\cdots \abs{\tau_\ell}) =	\Gr{k}(\abs{\tau_1\cdots \tau_\ell})
		      	=	\Gr{k}(\abs{w}).
  \end{align*}
 For the empty word we obtain $
   \Bnorm{\chi(\emptyword)} 	=	\Bnorm{1_B}	=	1	\stackrel{\ref{axiom_H_monotonic_in_k}}{=}	\Gr{k}(0) = \Gr{k}(\abs{\emptyword}). \qedhere
  $
 \end{proof}

 Note that we need the following proposition only for the monoids from Example \ref{ex:monoids} constructed over the field of complex numbers.
 However, an easy complexification argument yields an analogous statement for these monoids in the case $\K  = \R$.
 \begin{prop}												\label{prop_extension_Banach}
  Let $(\Sigma,\abs{\cdot})$ be a graded index set, $M=\FreeMonoid{\Sigma}$ or $M=\FreeCommMonoid{\Sigma}$ and fix $k\in\N$.
  Assume further, that $B$ is a commutative Banach algebra over $\C$.
  Then every $k$-controlled function $f\in\ellInfty{k}{\Sigma}{B}$ with $\norm{f}_{\ell^\infty_k}<1/2$ uniquely extends to a $k$-controlled character $\chi_f\in\ellInfty{k}{M}{B}$ with $\norm{\chi_f}_{\ell^\infty_k}\leq 1$ and the (nonlinear) extension operator  
  \[
   \func{\Psi_k}{		\oBallin{1/2}{\ellInfty{k}{\Sigma}{B}}{0}		}{	\ellInfty{k}{M}{B}	}{f}{\chi_f}
  \]
  between open subsets of Banach spaces is of class $C^\omega_\C$.
 \end{prop}
 \begin{proof}
  \textbf{Step 1: $\Psi_k$ is well-defined.}	
  The universal property of the free monoid (or the free commutative monoid, respectively) allows one to uniquely extend $\smfunc{f}{\Sigma}{B}$ to a character $\smfunc{\chi_f}{M}{B}$ 
  (using multiplicative continuation).
  By Lemma \ref{lem_controlled_functions_on_generators} it follows that this extension $\chi_f$ is $k$-controlled with operator norm $1$, if $f\in \oBallin{1}{\ellInfty{k}{\Sigma}{B}}{0}$. 
  Therefore, the extension operator $\Psi_k$ is well-defined. \medskip

  \textbf{Step 2: $\Psi_k$ is continuous.}
  Let $f_0\in \oBallin{1/2}{\ellInfty{k}{\Sigma}{B}}{0}$ and let $\epsilon>0$ be given. 
  We may assume that $\epsilon\leq1$. 
  Let $f\in \oBallin{1/2}{\ellInfty{k}{\Sigma}{B}}{0}$ such that $f_1:=f-f_0$ has norm $\norm{f_1}_{\ell^\infty_k}\leq\frac{\epsilon}{2}$.
  We will show that
  \[
   \norm{\chi_{f_0+f_1}-\chi_{f_0}}_{\ell^\infty_k}\leq \epsilon.
  \]
  Clearly it suffices to prove that for every word $w\in M$ of length $\ell\geq0$ we have
  \[
   \Bnorm{\chi_{f_0+f_1}(w)-\chi_{f_0}(w)}\leq \epsilon \cdot \Gr{k}(\abs{w}).
  \]
  For the empty word $w=\emptyword$ this holds since the left hand side of the inequality is zero as characters map the unit to the unit.
  So, let $w=\tau_1\cdots\tau_\ell$ be a word of length $\ell\geq1$.
  \begin{align*}
   &\Bnorm{\chi_{f_0+f_1}(w)-\chi_{f_0}(w)}	  	=	\Bnorm{\prod_{j=1}^\ell \left( f_0(\tau_j) + f_1(\tau_j)	\right) - \prod_{j=1}^\ell f_0(\tau_j) }
						\\	=&	\Bnorm{\sum_{\alpha \in \smset{0,1}^\ell} \prod_{j=1}^\ell f_{\alpha(j)}(\tau_j) - \prod_{j=1}^\ell f_{0}(\tau_j)}
							=	\Bnorm{\sum_{\substack{\alpha \in \smset{0,1}^\ell\\ \alpha \neq 0} } \prod_{j=1}^\ell f_{\alpha(j)}(\tau_j)}
						\\   	\leq&	\sum_{\substack{\alpha \in \smset{0,1}^\ell\\ \alpha \neq 0} } \prod_{j=1}^\ell \Bnorm{f_{\alpha(j)}(\tau_j)}
						   	\leq	\sum_{\substack{\alpha \in \smset{0,1}^\ell\\ \alpha \neq 0} }\ \prod_{j=1}^\ell \left(\norm{f_{\alpha(j)}}_{\ell^\infty_k} \Gr{k}(\abs{\tau_j})\right)
						\\	\leq&	\sum_{\substack{\alpha \in \smset{0,1}^\ell\\ \alpha \neq 0} } \left(\frac{\epsilon}{2}\right)^{\abs{\alpha^{-1}(1)}} \left(\frac{1}{2}\right)^{\abs{\alpha^{-1}(0)}} \prod_{j=1}^\ell  \Gr{k}(\abs{\tau_j})
							\stackrel{\textup{\ref{axiom_H_Multi}}}{\leq}	\sum_{\substack{\alpha \in \smset{0,1}^\ell\\ \alpha \neq 0} } \underbrace{\epsilon^{\abs{\alpha^{-1}(1)}}}_{\leq\epsilon}\cdot  \left(\frac{1}{2}\right)^\ell \Gr{k}(\abs{w})
						\\	=& 	(2^\ell-1) \epsilon \left(\frac{1}{2}\right)^\ell \Gr{k}(\abs{w})
						   	< 	\epsilon \cdot\Gr{k}(\abs{w}).
  \end{align*}

  \textbf{Step 3: $\Psi_k$ is $\C$-analytic. }
  For each $w\in M$, define the operator
  \[
   \func{\pi_w}{\ellInfty{k}{M}{B}}{B}{h}{h(w).} 
  \]
  Now $\pi_w$ is continuous by Proposition \ref{prop_properties_of_ellInftyLimit}(a) and the family
  $
   \Lambda := \setm{\phi\circ \pi_w}{\phi \in B' ; w\in M}
  $
  separates the points of $\ellInfty{k}{M}{B}$ as the topological dual $B'$ separates the points of $B$ by the Hahn-Banach Theorem. 
  Hence the continuous operator $\Psi_k$ will be $\C$-analytic by Proposition \ref{prop: avs:weaka} if $\pi_w \circ \Psi_k$ is $\C$-analytic. 
  If $w=\tau_1\cdots \tau_\ell \in M$, the map 
  \[
   \func{\pi_w\circ \Psi_k}{\ellInfty{k}{\Sigma}{B}}{B}{f}{f(\tau_1)\cdots f(\tau_\ell)}
  \]
  is a continuous $\ell$-homogeneous polynomial (in $f$), whence $\C$-analytic (cf.\ \cite{BS71b}). 
 \end{proof}

 We will use this proposition now to prove the analogue result in the (LB)-space case:
 \begin{prop}												\label{prop_extension_LB}
  Let $(\Sigma,\abs{\cdot})$ be a pure graded index set and let $M=\FreeMonoid{\Sigma}$ or $M=\FreeCommMonoid{\Sigma}$.
  Then every controlled function $f\in\ellInftyLimit{\Sigma}{B}$ has a unique extension to a controlled character $\chi_f\in\controlledChar{M}{B}\subseteq\ellInftyLimit{M}{B}$ and the (nonlinear) extension operator
  \[
   \func{\Psi}{\ellInftyLimit{\Sigma}{B}}{\ellInftyLimit{M}{B}}{f}{\chi_f}
  \]
  is of class $C^\omega_\K$, i.e.~it is $\K$-analytic.
 \end{prop}
 \begin{proof}
  Since $\Sigma$ is pure, we may apply part (c) of Proposition \ref{prop_properties_of_ellInftyLimit}. 
  Hence, the locally convex space $\ellInftyLimit{\Sigma}{B}$ is the union of the balls $\oBallin{1/2}{\ellInfty{k}{\Sigma}{B}}{0}$.
  So, every controlled function is $k$-controlled with norm $<1/2$ for a large enough $k \in \N$. 
  In particular, we can deduce from Proposition \ref{prop_extension_Banach} that every controlled function admits an extension which is also controlled, whence the extension operator makes sense.
  
  If $\K=\C$, the analyticity is a direct application of \cite[Theorem A]{dahmen2010} which ensures the complex analyticity of a function defined on the directed union of balls of the same radius (in this case $1/2$) if the function is complex analytic and bounded on each step which is ensured in this case by Proposition \ref{prop_extension_Banach}.

  For $\K=\R$ we consider the complexification $B_\C$ of $B$. Now $B_\C$ is a complex Banach algebra which contains $B$ as a closed real subalgebra whose norm is induced by the norm on $B_\C$ (see \cite[Section 2]{BS71a} and \cite{MR1688213}).
  For each $k\in\N$ the complex Banach space $\ellInfty{k}{\Sigma}{B_\C}$
   decomposes into a direct sum of two copies of the closed real spaces $\ellInfty{k}{\Sigma}{B}$. And since finite direct sums and direct limits commute \cite[Theorem 3.4]{MR1878717}, we conclude $\ellInftyLimit{\Sigma}{B_\C}$ is the complexification of $\ellInftyLimit{\Sigma}{B}$ and induces the correct topology on the closed real subspace $\ellInftyLimit{\Sigma}{B}$.
  
  Now the complex extension operator
  $
   \smfunc{\Psi_\C}{\ellInftyLimit{\Sigma}{B_\C}}{\ellInftyLimit{M}{B_\C}}
  $ 
  is complex analytic and restricting to the real subspaces we obtain the real operator $\smfunc{\Psi}{\ellInftyLimit{\Sigma}{B}}{\ellInftyLimit{M}{B}}$ 
  which is therefore real analytic (see Definition \ref{defn: real_analytic}).
\end{proof}

 \begin{prop}[Controlled characters form a submanifold]							\label{prop_characters_form_submanifold}
  Let $(\Sigma,\abs{\cdot})$ be a graded index set and let $M=\FreeMonoid{\Sigma}$ or $M=\FreeCommMonoid{\Sigma}$.
  Assume that $(\Sigma, \abs{\cdot})$ is pure, i.e.~that $M$ is connected.
  Then the set $ \controlledChar{M}{B}$ of controlled characters
  is a $C^\omega_\K$-submanifold of the locally convex space 
  $\ellInftyLimit{M}{B}$
  of all controlled functions on the monoid $M$.
  A global chart for $ \controlledChar{M}{B}$ is the $C^\omega_\K$-diffeomorphism
  \[
   \func{\res}{\controlledChar{M}{B}}{\ellInftyLimit{\Sigma}{B}}{\chi}{\chi|_\Sigma}.
  \]
 \end{prop}
 \begin{proof}
  Let $\Ind:=M\setminus\Sigma$ the set of all words with length different from $1$.
  Then by \cite[Theorem 3.4]{MR1878717} we have an isomorphism
  $
   \ellInftyLimit{M}{B} \cong \ellInftyLimit{\Sigma}{B} \times  \ellInftyLimit{\Ind}{B}
  $
  of locally convex spaces.
  Thus 
  $
   \func{\Psi}{\ellInftyLimit{\Sigma}{B}}{\controlledChar{M}{B}\subseteq \ellInftyLimit{\Sigma}{B}}{f}{\chi_f}
  $
  is $C^\omega_\K$ by Proposition \ref{prop_extension_LB}.
  Now, consider the projection map
  \[
   \func{\pi_\Ind}{\ellInftyLimit{M}{B}}{\ellInftyLimit{\Ind}{B}}{h}{h|_\Ind}
  \]
  which is continuous linear and hence the composition
  $
   \smfunc{\pi_\Ind\circ\Phi}{\ellInftyLimit{\Sigma}{B}}{\ellInftyLimit{\Ind}{B}}
  $
  is $C^\omega_\K$ as well  which implies that the graph of $\pi_\Ind\circ \Psi$ is a closed split submanifold of $\ellInftyLimit{\Sigma}{B}\times\ellInftyLimit{\Ind}{B}$ which is $C^\omega_\K$-diffeomorphic to the domain.
  Using the above identification, this graph becomes the set $\controlledChar{M}{B}$.  As the graph of an analytic function is an analytic submanifold (the argument in the setting of locally convex manifolds coincides with the standard argument for finite-dimensional manifolds, cf.\ \cite[Example 1.30]{MR2954043}), the assertion follows. 
 \end{proof}

 \section{Lie groups of controlled characters on Hopf algebras}										\label{section_controlled_Hopf}
   \newcommand{\CoMultH}{\Delta_\Hopf}
   \newcommand{\CoUnitH}{\epsilon_\Hopf}
   \newcommand{\AntipodeH}{S_\Hopf}
 In this section we construct the Lie group structure for groups of controlled characters of suitable Hopf algebras.
 To this end, we have to restrict our choice of Hopf algebras.
 
 \begin{defn}
  Consider a graded and connected Hopf algebra $\Hopf$ together with $\Sigma \subseteq \Hopf$. We call $(\Hopf, \Sigma)$ a \emph{combinatorial Hopf algebra} if $\Hopf$ is, as an associative algebra, isomorphic to a polynomial algebra $\K [\Sigma]$ or $\K \langle \Sigma \rangle$. 
 \end{defn}

 Observe that the isomorphism realising a combinatorial Hopf algebra as a polynomial algebra with generating set $\Sigma$ is part of the structure of the combinatorial Hopf algebra.
 Recall that any \emph{commutative} graded and connected Hopf algebra $\Hopf$ is isomorphic to a polynomial algebra $\K[\Sigma]$ by a suitable version of the Milnor--Moore theorem (see \cite[Theorem 3.8.3]{MR2290769}).
The term ``combinatorial Hopf algebra'' refers to the r\^{o}le of $\Sigma$ in our main examples: Usually $\Sigma$ will be a set of combinatorial objects (trees, partitions, permutations, etc.) and the Hopf algebra will encode combinatorial information.
 The estimates considered in the following are always understood with respect to the basis of the polynomial algebra $\K[\Sigma]$ (or $\K \langle \Sigma\rangle$). 
 
  
  \begin{ex}[{The shuffle Hopf algebra \cite{MR1231799}}]\label{ex: shuffle}
  For an alphabet $\mathcal{A} \neq \emptyset$ we consider the word monoid $\mathcal{A}^*$  and the algebra $\K \langle \mathcal{A}\rangle$.
  Fix $\rho \colon \mathcal{A} \rightarrow \N$ to induce a grading
  \begin{displaymath}
   | \cdot | \colon \K \langle \mathcal{A}\rangle \rightarrow \N,\quad  a_1 \ldots a_n \mapsto \sum_{i=1}^n \rho(a_i)
  \end{displaymath}
  Let $a,b$ be letters and $u,w$ be words. We define the \emph{shuffle product} recursively by
    \begin{displaymath}
    1 \shuffle w  = w \shuffle 1 = w, \quad  (au) \shuffle (bw) := a(u\shuffle (bw)) + b((au) \shuffle w)
    \end{displaymath}
  The (bilinear extensions of the) shuffle product and the deconcatenation of words
    \begin{displaymath}
     \Delta (a_1 \cdot a_2 \cdot \ldots \cdot a_n) = \sum_{i=0}^n a_1 \cdot a_2 \cdot \ldots \cdot a_i \otimes a_{i+1} \cdot a_{i+2} \cdot \ldots \cdot a_n 
    \end{displaymath}
 turns $\K \langle \mathcal{A}\rangle$ into a graded and connected $\K$-bialgebra (i.e.\ a Hopf algebra).
 However, an explicit formula for the antipode is available as $$S (a_1 a_2 \ldots a_n) = (-1)^n a_n a_{n-1} \ldots a_1, \quad \forall a_1,\ldots a_n \in \Sigma.$$
 We call the resulting commutative Hopf algebra $\text{Sh}_\K (\mathcal{A}, \rho)$ the \emph{shuffle Hopf algebra}.
 The shuffle Hopf algebra usually appearing in the literature (e.g.\ \cite{MR1231799}) is constructed with respect to $\rho \equiv 1$.
 
 Now assume that there is a total order ``$\leq$'' on the alphabet $\mathcal{A}$ and order words in $\mathcal{A}^*$ by the induced lexicographical ordering.
 Then one defines the ``Lyndon words'' as those words $w \in \mathcal{A}^*$ which satisfy that for every non trivial splitting $w=uv$ the condition $u< v$.
 Radford's theorem (cf.\ \cite[Section 6]{MR1231799} or \cite{MR1747062}) shows that the shuffle Hopf algebra is isomorphic to the polynomial algebra $\K [\Sigma]$, where $\Sigma$ is the set of Lyndon words in $\mathcal{A}^*$.
 Summing up, $(\text{Sh}_\K (\mathcal{A}, \rho), \Sigma)$ is a combinatorial Hopf algebra.

 Recall that shuffle Hopf algebras appear in diverse applications connected to numerical analysis, see e.g.\ \cite{MR2407032,MR3648103,MR3485151,MR2790315}.
 Also they appear in the work of Fliess in control theory (we refer to \cite[Section 6, Example (5)]{Oberst} for an account).
 Furthermore, there are generalisations of the shuffle Hopf algebra, e.g.\ the Quasi-shuffle Hopf algebras \cite{MR1747062}, which are also combinatorial Hopf algebras.
\end{ex}
 

 \begin{defn}									\label{defn_control_pair}			
  Let $(\Hopf, \Sigma)$ be a combinatorial Hopf algebra and $(\Gr{n})_{n\in \N}$ a growth family. 
  We call $((\Hopf, \Sigma), (\Gr{n})_{n\in \N})$ \emph{control pair} if $\smfunc{\CoMultH}{\Hopf}{\Hopf \otimes \Hopf}$ and $\smfunc{\AntipodeH}{\Hopf}{\Hopf}$ are $\ell^1_\leftarrow$-continuous with respect to the growth family $(\Gr{n})_{n\in \N}$ (see Lemma \ref{lem_continuity_of_linear_maps}).\footnote{In order for this to make sense, we use the identifications $\Hopf = \K^{\left(M\right)}$ and $\Hopf\otimes\Hopf = \K^{(M\times M)}$, where $M=\FreeCommMonoid{\Sigma}$ or $M=\FreeMonoid{\Sigma}$, respectively.}
 \end{defn}
 
In general $\ell^1_\leftarrow$-continuity of the multiplication with respect to a given growth family does not seem to be automatic (due to the inequality in \ref{axiom_H_Multi}). However, continuity of the multiplication is not needed in our approach. 
Finally, we address why well known recursion formulae for the antipode are in general not sufficient to establish $\ell^1_{\leftarrow}$-continuity of the antipode. 
 
 \begin{rem}\label{rem: recursion}
  For a graded and connected Hopf algebra $\Hopf$, the antipode can be compute recursively (cf.\ \cite[Corollary 5]{MR2523455}) as 
  \begin{align}
   S(x) &= -x - \sum_{(x)} S(x') x'' \label{rec:1}
        = -x - \sum_{(x)} x'S(x''). 
  \end{align}
 Naively one may hope to derive from this formula the $\ell^1$-continuity (with respect to the growth families from Proposition \ref{prop: growth:fam}) if one can establish the $\ell^1$-continuity of the coproduct first.
 However, this fails even in easy cases, where the $\ell^1$-continuity of the coproduct is almost trivial. 
 For example for the combinatorial Hopf algebra \cite{MR3223292} of (decorated) graphs, the antipode is given by  
  \begin{displaymath}
   S (\Gamma, D) = - (\Gamma, D) - \sum_{\emptyset \subsetneq (\Lambda, F) \subsetneq (\Gamma, D)} S (\Lambda, F) \otimes (\Gamma /\Lambda, D/F)).
  \end{displaymath}
  Here we hide all technicalities (e.g.\ on the graphs and the decorations). For us, the crucial observation is that all coefficients occurring are $\pm 1$ and thus it suffices to count summands in the recursion. 
  As $\Lambda$ is a subgraph of $\Gamma$ and the grading is by number of edges. Assume that $\Gamma$ has $n$ edges, then in the $k$th step of the recursion at most $2^{n-k}$ summands. 
  However, iterating this, we obtain the (very rough) estimate that in a worst case one has to account for $2^n2^{n-1} \cdots 2^{1} \leq 2^{\frac{n^2+n}{2}}$ summands. 
  In total, we thus obtain super exponential growth, whence we can not hope to obtain a control pair if we choose for example exponential growth.
  Summing up, even if the coproduct is very easy, the recursion formulae \eqref{rec:1} can not be used to derive $\ell^1$-continuity for $S$. 
  
  The deeper reason why the recursion formula is unsuitable for our purposes is that it contains in general an enormous amount of terms which cancel each other. 
  Thus a cancellation free, non-recursive formula would be desirable.
  In general however, it is a difficult combinatorial problem to construct such a formula. This problem has been dubbed the \emph{antipode problem} \cite[Section 5.4]{MR3077234} (also cf.\ \cite{MR3091063}).
  Recently a cancellation free formula, the so called preLie forest formula \cite[Theorem 8]{MP15} has been constructed for the so called \emph{right-handed Hopf algebras}, cf.\ \ref{setup:RHHopf}.
  Note that the Hopf algebras considered in the present paper are often of this type (e.g.\ the Connes-Kreimer algebra of rooted trees \ref{ex: CKHopf} and the Fa\`{a} di Bruno algebra (Example \ref{ex_faadiBruno}) are right-handed Hopf algebras). 
  Unfortunately, we were not able to establish the $\ell^1$-continuity of the antipode from the preLie forest formula in this general setting.
 \end{rem} 
 
 \begin{setup}\label{setup: niceHopf}
  Throughout this section, assume that $B$ is a commutative Banach algebra, $\left(	\Gr{k}	\right)_{k\in\N}$ is a fixed convex growth family, $(\Hopf, \Sigma)$ a combinatorial Hopf algebra.
  In addition, we assume that $((\Hopf, \Sigma), \left(	\Gr{k}	\right)_{k\in\N})$ is a control pair.
  As a vector space, $\Hopf$ is hence isomorphic to $\K^{(M)}$, where $M=\FreeCommMonoid{\Sigma}$ or $M=\FreeMonoid{\Sigma}$, respectively.
  
  A $\K$-linear map $\smfunc{\phi}{\Hopf}{B}$ is \emph{controlled} 
   if its restriction to $M$ is controlled (in the sense of \ref{setup_controlled_maps}), i.e.~if $\phi|_M \in \ellInftyLimit{M}{B}$. 
  By slight abuse of notation, we will denote the space of all controlled linear maps by $\ellInftyLimit{\Hopf}{B}$.
 \end{setup}

We briefly digress here to define a category of combinatorial Hopf algebras suitable for our purpose (meaning that the morphisms in this category mesh well with the notion of controlled functions). 

  \begin{setup}[The category of combinatorial Hopf algebras]
 A Hopf algebra morphism between combinatorial Hopf algebras $F \colon (\Hopf, \Sigma) \rightarrow (\widetilde{\Hopf}, \widetilde{\Sigma})$ will be called \emph{morphism of combinatorial Hopf algebras} if 
 \begin{itemize}
 \item it is of degree $0$, i.e.\ $F(\Hopf_n) \subseteq \widetilde{\Hopf}_n, \forall n\in \N_0$,
 \item there is $C>0$ constant, such that for every $\tau \in \Sigma$ with $F (\tau) = \sum_{\sigma \in \widetilde{\Sigma}} c_\sigma^\tau \sigma$ the estimate $\sum_{\sigma \in \widetilde{\sigma}} |c_\sigma^\tau| < C^{|\tau|}$ holds. 
 \end{itemize} 
 Combinatorial Hopf algebras and their morphisms form a category $\CoHopf$. The definition of morphism is geared towards preserving the controlled morphisms, i.e.\ if $((\Hopf, \Sigma), \left(	\Gr{k}	\right)_{k\in\N})$ and $((\widetilde{\Hopf}, \widetilde{\Sigma}), \left(	\Gr{k}	\right)_{k\in\N})$ are control pairs and $F \colon (\Hopf, \Sigma) \rightarrow (\widetilde{\Hopf}, \widetilde{\Sigma})$ is a morphim of combinatorial Hopf algebras, then for every controlled $\phi \colon \widetilde{\Hopf} \rightarrow B$ also $\phi \circ F$ is controlled. To see this, assume that $\varphi$ is in $\ell^\infty_k(\Hopf,B)$, then
 \begin{equation}\label{estimate}
 \lVert \phi \circ F\rVert_{\ell^\infty_k} = \sup_{\tau \in \Sigma} \frac{\lVert \phi (\sum_{\sigma \in \widetilde{\sigma}} c_\sigma^\tau \sigma) \rVert_B}{\omega_k (|\sigma|)} \leq \sup_{\tau \in \Sigma} \sum_{\sigma \in \widetilde{\Sigma}} |c^\tau_\sigma|  \frac{\lVert\phi (\sigma)\rVert_B}{\omega_k (|\tau|)} \leq \sup_{\tau \in \Sigma} C^{|\tau|} \lVert \varphi\rVert_{\ell^\infty_k}.\end{equation}
 The last inequality is due to $|\sigma|=|\tau|$ since $F$ is of degree $0$. Applying \ref{axiom_H_Infinity} we see that by replacing $k$ with some larger $K$, the right hand side of \eqref{estimate} is bounded, whence $\phi \circ F$ is controlled.
  
 Recently \cite[3.2]{CEMM18} proposed another definition for a category of combinatorial Hopf algebras. In general both categories are incomparable as the notion of combinatorial Hopf algebras are incomparable (though a key idea in both approaches is to single out a specific generating set). Note that neither the category in \cite[3.2]{CEMM18} nor our category of combinatorial Hopf algebras is a subcategory of the category of connected and graded Hopf algebras (as we will discuss in Section \ref{sect: examples} Hopf algebras which carry different combinatorial structures).  
 \end{setup}

 \begin{defn}							\label{def: character}
  A linear map $\smfunc{\phi}{\Hopf}{\lcB}$ is called 
  \begin{itemize}
   \item ($\lcB$-valued) \emph{character} if it is a homomorphism of unital algebras, i.e.
  \begin{equation}\label{eq char:char}
   \phi(a b) = \phi(a)\phi(b)\text{ for all } a,b \in \Hopf \text{ and } \phi(1_\Hopf)=1_\lcB.
  \end{equation}
  Equivalently, $\phi$ is a character of the underlying monoid $(\Hopf , m_\Hopf)$.
  \item  \emph{infinitesimal character} if
  \begin{equation}\label{eq: InfChar:char}
   \phi\circ m_\Hopf = m_\lcB\circ (\phi\otimes \epsilon_\Hopf + \epsilon_\Hopf \otimes \phi),
  \end{equation}
  which means for $a,b \in \Hopf$ that $\phi(a b)=\phi(a) \epsilon_\Hopf(b) + \epsilon_\Hopf(a) \phi(b)$.
  \end{itemize}
 The set of characters is denoted by $\Char{\Hopf}{\lcB}$, while $\InfChar{\Hopf}{\lcB}$ denotes the set of all infinitesimal characters.
 Further, we let $\controlledChar{\Hopf}{B} = \Char{\Hopf}{\lcB} \cap \ellInftyLimit{\Hopf}{B}$ be the set of controlled Hopf algebra characters and $\controlledInfChar{\Hopf}{\lcB} = \InfChar{\Hopf}{\lcB} \cap  \ellInftyLimit{\Hopf}{B}$ be the set of controlled infinitesimal characters.
 \end{defn}

 \begin{defn}[Convolution of controlled linear maps]
  Let $\phi,\psi\in \Hom(\Hopf,B)$ be two linear maps.
  We define the \emph{convolution} of the linear maps $\phi$ and $\psi$ to be
  \[
   \phi\star\psi := m_B \circ (\phi\otimes\psi)\circ \CoMultH.
  \]
  Here $\smfunc{m_B}{B\otimes B}{B}$ is the linear map induced by the algebra multiplication.

  The space $\Hom(\Hopf,B)$ of all linear maps from $\Hopf$ to $B$ is an associative unital algebra with the convolution product and the neutral element $\one \coloneq 1_B \cdot \CoUnitH$. 
 \end{defn}

 It is well known that the convolution turns the set of characters $\Char{\Hopf}{B}$ into a group (see e.g.\ \cite[Section 2]{BDS16}). 
 The definition of a growth family is geared towards turning the controlled functions into a subgroup of all characters. 
 To emphasise this point, we now return to the family of functions $(\Gr{k} (n) := e^{-\frac{n}{k}})_{k\in \N}$ discussed in Remark \ref{rem: nongrowthfamily}.
 For these functions the controlled characters, in general, do not form a group. 
 
  \begin{ex}\label{ex: counterex}
    Recall that the family of functions $(\Gr{k} (n) := e^{-\frac{n}{k}})_{k\in \N}$ discussed in Remark \ref{rem: nongrowthfamily} is not a growth family (it violates property \ref{axiom_H_Infinity}).
    We will now see that for this reason characters which are controlled with respect to this family of functions (which is not a growth family!) do not form a group.
    
    To this end, let $\Hopf= \K [X]$ be the polynomial algebra in one indeterminate and endow it with the coproduct defined on the generators as $\Delta (X^i) = \sum_{k=0}^i X^k \otimes X^{i-k}$ and the grading $|X^i|=i$.
    With this structure, $\Hopf$ becomes a graded, connected and commutative Hopf algebra (an antipode is given by $S(X^i)= (-1)^iX^i$ also cf.\ Remark \ref{rem: recursion}). 
    It is well known that in this case $(\Char{\K [X]}{\K} , \star) \cong (\K , +)$ as groups (see e.g.\ \cite[Section 1.4]{MR547117}).
    One easily computes that the subset of characters which are controlled with respect to the family of functions 
    $(\Gr{k} (n) := e^{-\frac{n}{k}})_{k\in \N}$ corresponds under the above isomorphism to  $\{ x \in \K \mid \abs{x} < 1\}$ which is not a subgroup of $(\K,+)$.
   \end{ex}

 \begin{prop}[Locally convex algebra of controlled maps]\label{prop: lcvx:alg}
  The space $\ellInftyLimit{\Hopf}{B}$ is a locally convex unital topological algebra with respect to convolution.
 \end{prop}
 \begin{proof}
  Note that the counit $\smfunc{\CoUnitH}{\Hopf}{\K}$ is automatically $\ell^1_\leftarrow$-continuous for a graded and connected Hopf algebra. 
  This follows immediately from Lemma \ref{lem_continuity_of_linear_maps} as for $\Hopf$ graded and connected 
   $\CoUnitH (x) =\begin{cases}
         r & \text{ if } x = r\one_\Hopf \in \K \one_\Hopf = \Hopf_0,\\
         0 & \text{ else.}
    \end{cases}$
  Hence the unit $\one$ is controlled, i.e.~$\one = 1_B\cdot \CoUnitH \in\ellInftyLimit{\Hopf}{B}$. 
  To prove that $\ellInftyLimit{\Hopf}{B}$ is closed under convolution, fix $k_1\in\N$ and consider $\phi,\psi\in\ellInfty{k_1}{\Hopf}{B}$.
  We construct $k_2\in\N$ such that $\phi\star\psi\in\ellInfty{k_2}{\Hopf}{B}$.
  Let $M$ denote the canonical basis of $\Hopf$, i.e.~$M=\FreeCommMonoid{\Sigma}$ or $M=\FreeMonoid{\Sigma}$ (see \ref{setup: niceHopf}).
  Since $\CoMultH$ is $\ell^1_\rightarrow$-continuous, we can choose $k_2 > k_1$ and $C>0$ with 
  \begin{equation}													\label{eq:comult}
     \ellOneNorm{\CoMultH (\tau)}{k_1}	\leq C \Gr{k_2}(|\tau|)	  \quad \forall \tau \in M.
  \end{equation}
  Let $\tau \in M$ be given. 
  Write $\Delta_\Hopf (\tau)$ as a linear combination of the basis of $\Hopf \otimes \Hopf$ 
  $
   \CoMultH(\tau) = \sum_{\mu,\sigma \in M} c_{\mu,\sigma} \mu\otimes\sigma.
  $
  and apply \eqref{eq:comult} to obtain:
  \begin{equation}													\label{eq:comult_explicitely}
   \sum_{\mu,\sigma} \abs{c_{\mu,\sigma}} \Gr{k_1}(\abs{\mu}+\abs{\sigma})  	\leq 	C \Gr{k_2}(|\tau|).
  \end{equation}
  This allows us the following computation:
  \begin{align*}
   \Bnorm{\phi\star\psi(\tau)}	  &	=	\Bnorm{ m_B\circ(\phi\otimes\psi)\circ\CoMultH (\tau)}
				   	=	\Bnorm{ \sum_{\mu,\sigma} c_{\mu,\sigma} \phi(\mu) \psi(\sigma)}
				\\&	\leq	\sum_{\mu,\sigma} \abs{c_{\mu,\sigma}} \Bnorm{   \phi(\mu)} \Bnorm{ \psi(\sigma)}
				\\&	\leq	\sum_{\mu,\sigma} \abs{c_{\mu,\sigma}} \norm{   \phi|_M}_{\ell^\infty_{k_1}}\Gr{k_1}(\abs{\mu}) \norm{   \psi|_M}_{\ell^\infty_{k_1}}\Gr{k_1}(\abs{\sigma})
				\\&\stackrel{\text{\ref{axiom_H_Multi}}}{\leq}	
						\norm{   \phi|_M}_{\ell^\infty_{k_1}}  \norm{   \psi|_M}_{\ell^\infty_{k_1}} \sum_{\mu,\sigma} \abs{c_{\mu,\sigma}} \Gr{k_1}(\abs{\mu}+\abs{\sigma})
				\\&\stackrel{\text{\eqref{eq:comult_explicitely}}}{\leq}	
						\norm{   \phi|_M}_{\ell^\infty_{k_1}}  \norm{   \psi|_M}_{\ell^\infty_{k_1}}  C \Gr{k_2}(\abs{\tau}).
  \end{align*}
 Divide both sides of the inequality by $\Gr{k_2}(\abs{\tau})$ and pass to the supremum over $\tau \in M$ to see that $\phi\star\psi$ is contained in $\ellInfty{k_2}{\Hopf}{B}$.
  By \cite[Theorem 3.4]{MR1878717} we have $\ellInftyLimit{M}{B} \times \ellInftyLimit{M}{B}= \lim_{\rightarrow} \left( \ellInfty{k}{M}{B} \times \ellInfty{k}{M}{B}\right)$ as locally convex spaces. 
  Hence the continuity of the convolution product follows from the continuity on the steps of the inductive limit (by \cite[Corollary 2.1]{dahmen2010}). 
 \end{proof}

 \begin{lem}\label{lem: char:mult}
 Consider the locally convex algebra $(\ellInftyLimit{\Hopf}{B},\star)$ (cf.\ Proposition \ref{prop: lcvx:alg}) and the associated Lie algebra $(\ellInftyLimit{\Hopf}{B},\LB{})$, where $\LB[\varphi , \psi] = \varphi \star \psi - \psi \star\varphi$ is the commutator bracket.
 \begin{enumerate}
   \item The controlled characters $\controlledChar{\Hopf}{B}$ form a closed subgroup of the unit group of the locally convex algebra .
  Inversion in this group is given by the map $\phi \mapsto \phi \circ S_\Hopf$ and the unit element is $1_\lcA\coloneq u_\lcB\circ\epsilon_\Hopf \colon \Hopf \rightarrow \lcB , x \mapsto \epsilon_\Hopf (x) 1_\lcB$.
  \item The controlled infinitesimal characters $\controlledInfChar{\Hopf}{\lcB}$ form a closed Lie subalgebra of $(\ellInftyLimit{\Hopf}{B},\LB{})$.
  \end{enumerate}
 \end{lem}
 \begin{proof}
  Characters form a group with respect to the convolution and pre-composition with the antipode. Similarly, the infinitesimal characters form a Lie Algebra. 
  We refer to \cite[4.3 Proposition 21 and 22)]{MR2523455} for proofs.
  For every controlled character, its inverse $\phi\circ S$ is again a controlled character since the antipode map $\smfunc{S}{\Hopf}{\Hopf}$ is $\ell^1_\leftarrow$-continuous by assumption. 
  It is clear from the definitions that $\controlledChar{\Hopf}{B}$ and $\controlledInfChar{\Hopf}{\lcB}$ are closed subsets with respect to the topology of pointwise convergence and hence Proposition \ref{prop_properties_of_ellInftyLimit}(a) implies the closedness with respect to the $\ell^\infty_\rightarrow$-topology.
 \end{proof}
   
  Now the first main result constructs the Lie group of controlled characters.
   
 \begin{thm}												\label{thm: controlled Lie group}
  Let $((\Hopf, \Sigma), (\Gr{n})_{n\in \N})$ be a control pair
  and let $B$ be a commutative Banach algebra. 
  The group of $B$-valued controlled characters $\controlledChar{\Hopf}{B}$ is an analytic Lie group modelled on the (LB)-space $\ellInftyLimit{\Sigma}{B}$.
 \end{thm}

 \begin{proof}
  We use the fact the $\Hopf$ is -- as an algebra -- isomorphic to $\K[\Sigma]$ or $\K\langle\Sigma\rangle$. Hence, by Proposition \ref{prop_characters_form_submanifold} the controlled characters form a closed split $\K$-analytic submanifold of the surrounding space $\ellInftyLimit{\Hopf}{B}$. 
  Since convolution is analytic (even continuous bilinear) on the surrounding space by Proposition \ref{prop: lcvx:alg}, it is still analytic when restricting to an analytic submanifold. 
  So, group multiplication is analytic. Since the inverse of a character is given by pre-composition with the antipode map, inversion on the group is the restriction of a continuous linear map on the surrounding space and hence analytic as well.
  This turns $\controlledChar{\Hopf}{B}$ into a $\K$-analytic Lie group. 
  As a manifold it is diffeomorphic to $\ellInftyLimit{\Sigma}{B}$ by Proposition \ref{prop_characters_form_submanifold}, so it is modelled on that space.
 \end{proof}
 
  Before we identify the Lie algebra of the group of $B$-valued controlled characters recall from  \cite[Theorem 2.8]{BDS16} the Lie group of $B$-valued characters.\smallskip
 
 \begin{setup}[\textbf{Lie group structure of $\Char{\Hopf}{B}$}]\label{setup:Chargp} \emph{
 Let $B$ be a commutative Banach algebra and $\Hopf$ be a graded and connected Hopf algebra (both over $\K$). 
 Then the character group $(\Char{\Hopf}{B},\star)$ of $B$-valued characters is a $\K$-analytic Lie group whose Lie algebra is $\InfChar{\Hopf}{B}$.
 The Lie group exponential is given by the exponential series} $$\exp_{\Char{\Hopf}{B}} \colon \InfChar{\Hopf}{B} \rightarrow \Char{\Hopf}{B} , \quad \phi \mapsto \sum_{n\geq 0} \frac{\phi^{\star_n}}{n!},$$
 \emph{where $\star_n$ is the $n$fold convolution of $\phi$ with itself. Finally, the canonical inclusion $\Char{\Hopf}{B} \rightarrow \textup{Hom}_\K (\Hopf, B)$ realises $\Char{\Hopf}{B}$ as a closed $\K$-analytic submanifold where the right hand side is endowed with the topology of pointwise convergence.}\smallskip
 \end{setup}
 Character groups of graded and connected Hopf algebras are infinite-dimensional Lie groups.
 The controlled character groups just constructed are naturally subgroups of the character groups and the inclusion turns out to be a Lie group morphism.
  
 \begin{lem}\label{lem: sl_in_full}
  Let $((\Hopf, \Sigma), (\Gr{n})_{n\in \N})$ be a control pair and $\controlledChar{\Hopf}{B}$ the Lie group induced by it.
  The canonical inclusion $\iota_{\Hopf, B} \colon \controlledChar{\Hopf}{B} \rightarrow \Char{\Hopf}{B},\ \phi \mapsto \phi$ is a $\K$-analytic morphism of Lie groups.
 \end{lem}
 
 \begin{proof}
  Obviously $\iota_{\Hopf,B}$ is a group morphism as $\controlledChar{\Hopf}{B}$ is a subgroup of $\Char{\Hopf}{B}$.
  Hence we only have to prove that it is $\K$-analytic.
  As outlined in \ref{section_controlled_Hopf}, the Hopf algebra $\Hopf$ is -- as an algebra -- either $\K[\Sigma]$ or $\K\langle \Sigma\rangle$ for some graded index set $\Sigma$.
  Let $M$ be $\FreeMonoid{\Sigma}$ or $\FreeCommMonoid{\Sigma}$, respectively.
  Thus with respect to the topology of pointwise convergence we have an isomorphism of topological vector spaces $\textup{Hom}_\K (\Hopf , B) \rightarrow B^M$ (compare \ref{setup_linear_maps_defined_on_a_basis}).
  Proposition \ref{prop_properties_of_ellInftyLimit} (a) shows that the inclusion $I \colon \ellInftyLimit{\Hopf}{B} \rightarrow B^M \cong \textup{Hom}_\K (\Hopf,B)$ is continuous linear, whence $\K$-analytic. 
  As $\iota_{\Hopf,B}$ arises as the (co-)restriction of $I$ to the closed analytic submanifolds $\controlledChar{\Hopf}{B} \subseteq \ellInftyLimit{\Hopf}{B}$  and $\Char{\Hopf}{B} \subseteq \textup{Hom}_\K (\Hopf,B)$ it is again $\K$-analytic.
 \end{proof}
 
 Note however, that the topology of the controlled character groups is finer than the subspace topology induced by the full character group. 
 Thus the Lie group structure of the controlled character groups does not turn them into Lie subgroups of the full character group. 
 Using the Lie group morphism just constructed we can now identify the Lie algebra as the Lie algebra of all controlled infinitesimal characters with the following natural Lie bracket.
 
 \begin{prop}															\label{prop_Lie_algebra_of_controlled_characters}
  The Lie algebra of the Lie group $\controlledChar{\Hopf}{B}$ induced by the control pair $((\Hopf, \Sigma), (\Gr{n})_{n\in \N})$, is given by $(\controlledInfChar{\Hopf}{B},\LB )$, where the Lie bracket is the commutator bracket induced by the convolution
    \begin{displaymath}
    \LB[\varphi , \psi] = \varphi \star \psi - \psi \star\varphi.
    \end{displaymath}
  As a locally convex vector space $\controlledInfChar{\Hopf}{B}$ is isomorphic to $\ellInftyLimit{\Sigma}{B}$ via 
  \[
   \func{\res}{\controlledInfChar{\Hopf}{B}}{\ellInftyLimit{\Sigma}{B}}{\psi}{\psi|_\Sigma}
  \]
  whose inverse extends a function $f\in\ellInftyLimit{\Sigma}{B}$ to all words by assigning each word of length different from $1$ to zero.
 \end{prop}
 
 \begin{proof}
  By Lemma \ref{lem: sl_in_full},  $\iota_{\Hopf, B} \colon \controlledChar{\Hopf}{B} \rightarrow \Char{\Hopf}{B},\ \chi \mapsto \chi$ is a $\K$-analytic morphism of Lie groups. 
  Hence $\Lf (\iota_{\Hopf,B}) := T_{\one} \iota_{\Hopf,B} \colon \controlledInfChar{\Hopf}{B} \rightarrow \InfChar{\Hopf}{B}$ is a morphism of Lie algebras.
  Now as $\iota_{\Hopf,B}$ arises as the restriction of a continuous linear map $I$ to a $\K$-analytic submanifold, we deduce that  $T_{\one} \iota_{\Hopf,B}$ is the restriction of $T_{\one} I$ which can be identified again with $I$.
  In conclusion $\Lf(\iota_{\Hopf,B}) (\varphi) = \varphi$ for every $\varphi \in  \controlledChar{\Hopf}{B}$ by definition of $I$.
  Hence $\LB = \Lf(\iota_{\Hopf,B}) \circ \LB  = \LB[\Lf(\iota_{\Hopf,B}) (\cdot), \Lf(\iota_{\Hopf,B})(\cdot)]_{\InfChar{\Hopf}{B}} = \LB_{\InfChar{\Hopf}{B}}$.
  Summing up the assertion follows from the formula for the Lie bracket of $\InfChar{\Hopf}{B}$.

  It remains to show the statement about the isomorphism. 
  By Proposition \ref{prop_characters_form_submanifold} we know that the $C^\omega_\K$-manifold $\controlledChar{\Hopf}{B}$ is diffeomorphic to $\ellInftyLimit{\Sigma}{B}$ via the diffeomorphism
  \[
   \nnfunc{\controlledChar{\Hopf}{B}}{\ellInftyLimit{\Sigma}{B}}{\phi}{\phi|_\Sigma.}
  \]
  Taking the tangent map at the identity of $\controlledChar{\Hopf}{B}$ we obtain an isomorphism of the corresponding tangent spaces $\controlledInfChar{\Hopf}{B}$ and $\ellInftyLimit{\Sigma}{B}$. Since the diffeomorphism can be extended to the continuous linear map 
  \[
   \nnfunc{\ellInftyLimit{\Hopf}{B}}{\ellInftyLimit{\Sigma}{B}}{\phi}{\phi|_\Sigma}
  \]
  its tangent map at each point is also given by this exact formula.
 \end{proof}

 In Theorem \ref{thm: controlled Lie group} we have seen that the controlled characters of a real Hopf algebra form an analytic real Lie group. 
 We will now establish that this group admits a complexification in the following sense (cf.\ \cite[9.6]{HG15reg}).
 
 \begin{defn}  \label{defn: complexification}
 Let $G$ be a real analytic Lie group modelled on the locally convex space $E$ and $G_\C$ be a complex analytic Lie group modelled on $E_\C$.
 Then $G_\C$ is called a \emph{complexification} of $G$ if $G$ is a real submanifold of $G_\C$, the inclusion $G \rightarrow G_\C$ is a group homomorphism and for each $g \in G$, there exists an open $g$-neighbourhood $V \subseteq G_\C$ and a complex analytic diffeomorphism $\phi \colon V \rightarrow W \subseteq E_\C$ such that $\phi(V\cap G) = W\cap E$.
 \end{defn}

  Note that for a real Hopf algebra $\Hopf$, also its complexification $\Hopf_\C$
  is a Hopf algebra.
  If $((\Hopf, \Sigma), (\Gr{n})_{n\in \N})$ is a control pair then, since $\Hopf$ is isomorphic to $\R[\Sigma]$ or $\R\langle \Sigma \rangle$, clearly $\Hopf_\C$ is isomorphic to $\C[\Sigma]$ or $\C\langle \Sigma \rangle$, respectively.
  For a commutative Banach algebra $B$ its complexification $B_\C$ is again a Banach algebra (see \cite[Lemma 2]{MR0273396} such that its norm coincides on $B$ with the original norm. From Lemma \ref{lem_continuity_of_linear_maps} we deduce thus:
   
  \begin{lem}\label{lem: triple:compl}
   If $((\Hopf , \Sigma), (\omega_n)_{n \in \N})$ is a control pair then $((\Hopf_\C , \Sigma), (\omega_n)_{n \in \N})$ is a control pair.
  \end{lem}
 
  Armed with this knowledge, we can now generalise the complexification of the (tame) Butcher group discussed in \cite[Corollary 2.8]{BS16}. 
 
 \begin{prop}								\label{prop: complexification}
  Let $B$ be a commutative real Banach algebra and $((\Hopf , \Sigma), (\omega_n)_{n \in \N})$ be a control pair such that $\Hopf$ is a real Hopf algebra. Denote by $\Hopf_\C$ and $B_\C$ their complexifications. 
  Then the complex Lie group
  \begin{itemize}
   \item [\textup{(a)}] $\controlledChar{\Hopf_\C}{B_\C}$ constructed in Theorem \ref{thm: controlled Lie group} is the complexification of the real Lie group $\controlledChar{\Hopf}{B}$.
   \item [\textup{(b)}] $\Char{\Hopf_\C}{B_\C}$ (cf.\ \ref{setup:Chargp}) is the complexification of the real Lie group $\Char{\Hopf}{B}$.
  \end{itemize}
  Summing up, the following diagram is commutative: 
  \begin{equation}\label{diag: complexification} \begin{aligned}
   \begin{xy}
  \xymatrix{
     \controlledChar{\Hopf}{B} \ar[rr]^-{\subseteq} \ar[d]_{\iota_{\Hopf, B}}   & &  \controlledChar{\Hopf_\C}{B_\C} \ar[d]^{\iota_{\Hopf_\C, B_\C}}  \\
     \Char{\Hopf}{B} \ar[rr]^-{\subseteq}         &    &   \Char{\Hopf_\C}{B_\C}   
  }
   \end{xy}
   \end{aligned}
  \end{equation}
 \end{prop}

 \begin{proof}
  \begin{itemize}
   \item[(a)] The complexification of the modelling space $E:= \ellInftyLimit{\Sigma}{B}$ of $\controlledChar{\Hopf}{B}$ is $\ellInftyLimit{\Sigma}{B_\C}$ by \cite[Theorem 3.4]{MR1878717}, since Proposition \ref{prop_extension_LB} shows that the complexifications of the steps in the limit are $\ellInfty{k}{\Sigma}{B_\C}$.
   In the canonical global charts one immediately sees that the inclusion $\controlledChar{\Hopf}{B} \subseteq \controlledChar{\Hopf}{B}$ is a group morphism which realises $\controlledChar{\Hopf}{B}$ as a real analytic submanifold of $\controlledChar{\Hopf}{B}$.
   Summing up, $\controlledChar{\Hopf_\C}{B_\C}$ is the complexification of $\controlledChar{\Hopf}{B}$ in the sense of Definition \ref{defn: complexification} (taking $\phi$ to be the canonical chart).
   \item[(b)]
   Recall from \cite[Theorem 2.7]{BDS16} that the Lie group $\Char{\Hopf}{B}$ is a closed analytic submanifold of the densely graded algebra $(\text{Hom}_\R (\Hopf, B), \star)$.
   Here $(\text{Hom}_\R (\Hopf, B) \cong B^{\mathcal{I}}$, where the right hand side is endowed with the product topology and $\mathcal{I}$ is some choice of a vector space basis for $\Hopf$. 
   Note that $\mathcal{I}$ is then also a basis for the complex vector space $\Hopf_\C$.
   Obviously, the canonical inclusion $B^{\mathcal{I}} \subseteq B_\C^{\mathcal{I}} \cong \text{Hom}_\C (\Hopf_\C, B_\C)$ is a complexification, whereas $\Char{\Hopf}{B} \rightarrow \Char{\Hopf_\C}{B_\C}$ is one.  
   \end{itemize}
  Obviously, the diagram \eqref{diag: complexification} is commutative by construction and all morphisms are ($\R / \C$-)analytic morphisms of Lie groups.
 \end{proof}

 \begin{rem}
  Part (b) of Proposition \ref{prop: complexification}, is true for arbitrary connected and graded Hopf algebras, i.e.\ they need not be isomorphic (as associative algebras) to $\K [\Sigma]$ or $\K \langle \Sigma \rangle$. 
  Further, the algebra $B$ can be any (real) locally convex algebra.
 \end{rem}
 
 We end this section with some comparison results concerning the groups of controlled characters of a combinatorial Hopf algebra with respect to different growth families.
 \newcommand{\GrTilde}[1]{\widetilde\omega_{#1}}
 \begin{prop}
  Let $(\Hopf, \Sigma)$ be a combinatorial Hopf algebra. Fix two growth families such that $((\Hopf, \Sigma), (\Gr{k})_{k\in \N})$ and $((\Hopf, \Sigma), (\GrTilde{k})_{k\in \N})$ are control pairs.
  For a commutative Banach algebra $B$ denote by $\controlledChar{\Hopf}{B}$ (resp.\ $\widetilde{G}_{\mathrm{ctr}}(\Hopf, B) $ ) the controlled characters with respect to $(\Gr{k})_{k\in \N}$ (resp.\ $(\GrTilde{k})_{k\in \N}$).
  If for every $r \in \N$ there exists $s \in \N$ with $\Gr{r}(n) \leq \GrTilde{s}(n) \quad \forall n\in \N$ then we obtain an analytic Lie group morphism 
  \begin{displaymath}
                                \controlledChar{\Hopf}{B} \rightarrow \widetilde{G}_{\mathrm{ctr}}(\Hopf, B) , \phi \mapsto \phi.                                                                                      
  \end{displaymath}
 \end{prop}
 \begin{proof}
  The modelling space of the Lie group $\controlledChar{\Hopf}{B}$ is the inductive limit of the Banach spaces $\ell^\infty_r (\Sigma, B)$. 
  To distinguish the weights, we denote by $\ell^\infty_{r, \GrTilde{k}} (\Sigma, B)$ the steps of the inductive limit forming the modelling space of $\widetilde{G}_{\mathrm{ctr}}(\Hopf, B) $.
  Now the condition $\Gr{r}(n) \leq \GrTilde{s}(n) \quad \forall n\in \N$ that there is a canonical inclusion of $\ell^\infty_r (\Sigma, B) \rightarrow \ell^\infty_{s, \GrTilde{k}} (\Sigma, B)$ for every $r\in \N$.
  These maps are continuous linear and composing with the limit maps of the steps $\ell^\infty_{s, \GrTilde{k}}$ we obtain continuous linear inclusions of the steps $\ell^\infty_r (\Sigma, B)$ into the modelling space of $\widetilde{G}_{\mathrm{ctr}}(\Hopf, B) $.
  By the universal property of the inductive limit, we thus obtain a continuous linear inclusion of the modelling space of $\controlledChar{\Hopf}{B}$ into the one of $\widetilde{G}_{\mathrm{ctr}}(\Hopf, B) $.
  Composing with the global charts of the groups of controlled characters, we obtain exactly the analytic Lie group morphism described in the proposition.
 \end{proof}

\section{Regularity of controlled character groups}

In this chapter we discuss regularity in the sense of Milnor for groups of controlled characters. 
Our approach here follows the strategy outlined in \cite[Section 4]{BS16}: 

First recall from \cite[Theorem B]{BDS16} that the full character group $\Char{\Hopf}{B}$ is $C^0$-regular. 
Hence to establish semiregularity, we prove that solutions of the differential equations in the full character group factor through $\controlledChar{\Hopf}{B}$ for suitable initial data.
We then prove that under certain assumptions one obtains smooth solutions of the equations on the controlled character group.

In a second step, we use then inductive limit techniques to establish the smoothness of the evolution operator. 
As a consequence the controlled character groups will even be regular.
However, certain assumptions are necessary to obtain the estimates used in our strategy. 
These methods do not allow us to establish a general regularity result as it is very hard to obtain estimates for general combinatorial Hopf algebras.
Since we also want to establish regularity properties for the controlled character groups of complex Hopf algebras a few conventions and remarks are needed:

Complex analytic maps are smooth with respect to the underlying real
 structure by \cite[Proposition 2.4]{hg2002a},
 whence for a complex Hopf algebra $\Hopf$ the complex Lie group $\controlledChar{\Hopf}{B}$ also carries the structure of a real Lie group.
 A complex Lie group is called regular, if the underlying real Lie group is regular (in the sense of the introduction).

 \begin{setup}\label{setup:RHHopf}
  Recall that the coproduct in a graded and connected Hopf algebra can be written as $\Delta (x) = \one_\Hopf \otimes x + x \otimes \one_\Hopf + \overline{\Delta}(x)$, where $\overline{\Delta} \colon \Hopf \rightarrow \bigoplus_{n,m \in \N} \Hopf_n \otimes \Hopf_m$ denotes the reduced coproduct.  
  Following \cite{MP15} a combinatorial Hopf algebra $(\Hopf, \Sigma)$ is called \emph{right-handed combinatorial Hopf algebra} if its reduced coproduct satisfies 
     \begin{displaymath}
      \overline{\Delta} (\Hopf) \subseteq \K^{(\Sigma)} \otimes \Hopf, \quad \text{with }  \K^{(\Sigma)} \text{ the vector space with base } \Sigma
      \end{displaymath}
  One can prove that (commutative) right-handed combinatorial Hopf algebras are closely connected to preLie algebras and the antipode is given by a Zimmermann type forest formula (cf.\ \cite[Theorem 8]{MP15}). 
 \end{setup}

 Hopf algebras occuring in the renormalisation of quantum field theories are typically right-handed Hopf algebras (such as the Connes-Kreimer Hopf algebra, Example \ref{ex: CKHopf}). 
 
In the following we fix a control pair $((\Hopf, \Sigma), (\Gr{n})_{n\in \N})$ with convex growth family and assume that $(\Hopf, \Sigma)$ is an right-handed combinatorial Hopf algebra.  
  The Lie group of controlled characters $\controlledChar{\Hopf}{B}$ will always be constructed with respect to such a pair, and some commutative Banach algebra $B$.
 Before we begin, recall the type of differential equation we wish to solve.

 \begin{setup}[Lie type differential equations on controlled character groups]
 As discussed in the introduction we are interested in solutions of the differential equation 
 \begin{displaymath}
  \begin{cases}
   \gamma'(t)&= \gamma(t).\eta(t)\\ \gamma(0) &= \one
  \end{cases}
 \end{displaymath}
  where $\eta \in C ([0,1],\controlledInfChar{\Hopf}{B})$ and the dot means right multiplication in the tangent Lie group $T\controlledChar{\Hopf}{B}$.   
 By Lemma \ref{lem: sl_in_full} the map $\iota_{\Hopf, B} \colon \controlledChar{\Hopf}{B} \rightarrow \Char{\Hopf}{B},\ \phi \mapsto \phi$ is a Lie group morphism.
 Its derivative $\Lf (\iota_{\Hopf, B}) \colon \Lf (\controlledChar{\Hopf}{B}) \rightarrow \Lf (\Char{\Hopf}{B})$ is the canonical inclusion of $\ellInftyLimit{\Hopf}{B}$ into $B^\Hopf$.
 As a consequence of \cite[Lemma 10.1]{HG15reg} if $\eta \in C^k ([0,1] , \controlledInfChar{\Hopf}{B})$ admits a $C^{k+1}$-evolution $\gamma\colon [0,1] \rightarrow \controlledChar{\Hopf}{B}$, then
  \begin{equation}\label{eq: evol:eq}
   \iota_{\Hopf, B} \circ \gamma = \Evol_{\Char{\Hopf}{B}} (\Lf (\iota_{\Hopf, B} ) \circ \eta).
  \end{equation}
  Thus from \cite[Step 1 in the proof of Theorem 2.11]{BDS16} we see that the differential equation for regularity can be rewritten for $\eta \in C([0,1], \controlledInfChar{\Hopf}{B})$ as 
  \begin{equation}\label{eq: LTR}
   \begin{cases}
   \gamma'(t)&= \gamma(t) \star \eta(t)\\ \gamma(0) &= \one
  \end{cases}
  \end{equation}
  (this can also be deduced directly from the submanifold structure).
  \end{setup}

  We will now prove that for certain Hopf algebras, the solution of the equation \eqref{eq: LTR} factors through the controlled character group if the initial curve $\eta$ takes its values in the Lie algebra of controlled characters.
  However, our proof depends on a certain estimate on the growth of the convolution product (we discuss this in Remark \ref{rem: lineargrowth} below). To formulate the growth estimate, we define the elementary coproduct.
  
  \begin{defn}
  For a right-handed Hopf algebra $(\Hopf, \Sigma)$ define 
  \begin{displaymath}
   \elcopro (\tau) := \sum_{\substack{
                                   \alpha , \beta \in \Sigma
                                  }} c_{\alpha ,\beta, \tau}\alpha \otimes \beta \quad \tau \in \Sigma.
  \end{displaymath}
  the \emph{right elementary coproduct}. Here the right-hand side is given by the terms in the basis expansion of the reduced coproduct $\Delta$ of the Hopf algebra. 
  Since the Hopf algebra is right-handed, the elementary coproduct restricts the terms only by imposing a condition on the second component of the tensor products. 
  Note that the right elementary coproduct will in general contain less terms than the reduced coproduct since we restrict in the second component to elements in the alphabet $\Sigma$.
  \end{defn}
 
 \begin{defn}[(RLB) Hopf algebra]
  We call a right-handed combinatorial Hopf algebra $(\Hopf, \Sigma)$ \emph{right-hand linearly bounded} (or \emph{(RLB)}) \emph{Hopf algebra} if there are constants $a,b > 0$ such that the elementary coproduct satisfies a uniform $\ell^1$-estimate 
 \begin{equation}\label{eq: linear:est}
   \norm{\elcopro  (\tau)}_{\ell^1} = \sum_{\substack{
                                  |\alpha | + |\beta| = |\tau| \\  \alpha , \beta \in \Sigma
                                  }} |c_{\alpha ,\beta, \tau}| \leq (a|\tau|+b) \quad \forall \tau \in \Sigma.
  \end{equation}
 \end{defn}
 For example, the Connes-Kreimer algebra (see Example \ref{ex: CKHopf}) is a (RLB) Hopf algebra, whereas the Fa\`{a} di Bruno algebra (cf.\ Example \ref{ex_faadiBruno}) is not of this type.
 We need a version of Gronwall's inequality (see \cite[Lemma 2.7]{MR2961944}, \cite[Lemma 1.1.24(a)]{dahmen2011}):
 \begin{lem}[Gronwall]															\label{lem_gronwall}
  Let $\smfunc{h}{[0,1]}{\left[ 0,+\infty \right[}$ be continuous, $A,B\geq0$ constant with
  \[
   h(t)\leq A + B \int_0^t h(s) \di s\quad \text{ for all }t\in[0,1].
  \]
  Then
  $
   h(t) \leq A e^{tB}\quad \text{ for all }t\in[0,1].
  $
 \end{lem}

 \begin{prop}													\label{prop_semiregularity}
  If $(\Hopf, \Sigma)$ is a (RLB) Hopf algebra and $((\Hopf, \Sigma), (\Gr{n})_{n\in \N})$ is a control pair, then $\controlledChar{\Hopf}{B}$ is $C^0$-semiregular.
 \end{prop}
 
 \begin{proof}
  Recall from \cite{BDS16} that the differential equation \eqref{eq: LTR} has a unique solution in $\Char{\Hopf}{B}$ (since this group is $C^0$-semiregular).  
  We will now prove that for every fixed $\eta \in C([0,1],\controlledInfChar{\Hopf}{B})$ the solution $\gamma \colon [0,1] \rightarrow \Char{\Hopf}{B}$ to \eqref{eq: LTR} in $\Char{\Hopf}{B}$ factors through the controlled characters.
  Since $\eta ([0,1]) \subseteq \controlledInfChar{\Hopf}{B}$ is compact and the Lie algebra is a compactly regular inductive limit (cf.\ Lemma \ref{lem_creg_lim}), we deduce from Proposition \ref{prop_properties_of_ellInftyLimit} (c) that there is $k \in \N$ with $\eta ([0,1]) \subseteq \oBallin{1}{\ell^\infty_k}{0}$. 
  Now define 
  \begin{displaymath} 
   h_n (t) := \sup_{\tau \in \Sigma, |\tau| \leq n} \frac{\norm{\gamma (t)(\tau)}}{\Gr{k} (|\tau|)} \quad \quad , t \in [0,1], \ n \in \N
  \end{displaymath}
  Our aim is to find an upper bound for $h_n$ growing at most exponentially in $n$ (and thus slower then any growth family).
    Choose $a,b>0$ such that \eqref{eq: linear:est} is satisfied.
  Recall that $\gamma' = \gamma \star \eta \in \Hom (\Hopf, B)$ by \eqref{eq: LTR}.
  Hence we can integrate this formula and use the properties of (infinitesimal) characters to obtain the following estimate for $n\in \N$: 
  \begin{align*}
   h_n (t) &\stackrel{\hphantom{\eqref{eq: InfChar:char} , \eqref{eq char:char}}}{=} \sup_{|\tau| \leq n} \frac{\norm{\gamma (t)(\tau)}}{\Gr{k} (|\tau|)}  
										  = \sup_{|\tau| \leq n} \frac{\norm{\int_0^t \gamma (s) \star \eta(s) \di s}}{\Gr{k} (|\tau|)}\\ 
	&\stackrel{\eqref{eq: InfChar:char} , \eqref{eq char:char}}{=} \sup_{|\tau| \leq n} \frac{1}{\Gr{k} (|\tau|)} \left\Vert \int_0^t \Bigl(  \gamma (s)(\tau) \underbrace{\eta (s)(\emptyset)}_{=0} + \underbrace{\gamma (s)(\emptyset)}_{=1_B} \eta (s)(\tau) + \right. \\
	& \left.\vphantom{\int_0^t}\hspace{4cm}+ \sum_{\substack{|\alpha | + |\beta| = |\tau| \\  \alpha \in \Sigma, \beta \in \Sigma^* \setminus \{\emptyset\}}} c_{\alpha, \beta,\tau} \gamma (s)(\alpha) \eta(s) (\beta) \Bigr) \di s\right\Vert \\
	&\stackrel{\eqref{eq: InfChar:char} , \eqref{eq char:char}}{=}  \sup_{|\tau| \leq n}  \frac{1}{\Gr{k} (|\tau|)}  \int_0^t \Bigl( \underbrace{\norm{\eta (s)(\tau)}}_{\leq \Gr{k} (|\tau|)}\di s +\\ 
	&\hspace{2cm} + \int_0^t\sum_{\substack{|\alpha |  + |\beta| = |\tau|,\\ \alpha, \beta \in \Sigma }}\hspace{-1em}  |c_{\alpha, \beta,\tau}| \underbrace{\norm{\gamma (s) (\alpha)}}_{\leq h_n (s) \omega_{k} (|\alpha|)} \underbrace{ \norm{\eta(s) (\beta)}}_{\leq \omega_k (|\beta|)} \Bigr) \di s   \\
	&\stackrel{\hphantom{\eqref{eq: InfChar:char} , \eqref{eq char:char}}}{\leq} 1 + \sup_{|\tau| \leq n} \sum_{\substack{|\alpha |  + |\beta| = |\tau|,\\ \alpha, \beta \in \Sigma}} |c_{\alpha, \beta,\tau}|  \frac{\omega_k (|\beta|) \omega_k (|\alpha|)}{\Gr{k} (|\tau|)} \int_0^t h_n (s) \di s \\
	&\stackrel{\eqref{eq: linear:est}}{\leq} 1 + \sup_{|\tau| \leq n} (a|\tau| +b) \int h_n(s) \di s \leq 1 + (an+b) \int_0^t h_{n}(s) \di s
  \end{align*}
  Working in a right-handed Hopf algebra ensured that there are no higher powers of $h_n$ in the formulae.
  By the Gronwall inequality (Lemma \ref{lem_gronwall}) this leads to 
  \begin{displaymath}
   h_n (t) \leq 1\cdot e^{(an+b)t} < \left(2^n\right)^{2a+2b}.
  \end{displaymath}
  By $(2a+2b)$ times applying \ref{axiom_H_Infinity} we obtain a $k_2>k$ such that 
  \[
   \Gr{k}(n) \left(2^n\right)^{2a+2b}\leq \Gr{k_2}(n)\quad \text{ for all }n\in\N.
  \]
  For a given $t\in[0,1]$ and $\tau\in \Sigma$ we deduce
  \begin{align*}
   \norm{\gamma(t)(\tau)}	&	=	\Gr{k}(\abs{\tau}) h_{\abs{\tau}}(t)
 			         	\leq	\Gr{k}(\abs{\tau})\left(2^{\abs{\tau}}\right)^{2a+2b}
  			         	\leq	\Gr{k_2}(\abs{\tau}).
  \end{align*}
  This shows that $\gamma$ takes only values in the (unit ball of the) space $\ellInfty{k_2}{\Hopf}{B}$.
  Now apply \ref{axiom_H_Infinity} again to find $k_3 > k_2$ with $(an+b)\Gr{k_2}(n) < \Gr{k_3} (n)$ for all $n \in \N$.
  Then $\smfunc{\gamma}{[0,1]}{ \oBallin{1}{\ellInfty{k_3}{\Hopf}{B}}{0} }$ is continuous due to the following estimate:
  \begin{align*}
   &\norm{\gamma(t)-\gamma(t_0)}_{\ell^\infty_{k_3}}	  	\leq 	\int_{t_0}^t \norm{\gamma'(s) }_{\ell^\infty_{k_3}}\di s 	=	\int_{t_0}^t \norm{\gamma(s)\star\eta(s) }_{\ell^\infty_{k_3}}\di s\\
    \leq&	 \sup_{\tau \in \Sigma} \int_{t_0}^t \left(\underbrace{\norm{\eta(s)(\tau)}_{\ell^\infty_{k_3}}}_{\leq 1} + \sum_{\substack{|\alpha |  + |\beta| = |\tau|,\\ \alpha, \beta \in \Sigma }} |c_{\alpha,\beta, \tau}| \frac{\Gr{k_2} (|\tau|)}{\Gr{k_3} (|\tau|)}\underbrace{\norm{\gamma(s) }_{\ell^\infty_{k_2}}}_{\leq 1} \underbrace{\norm{\eta(s) }_{\ell^\infty_{k_2}}}_{\leq 1} \right)\di s\\
    \leq &	\sup_{\tau \in \Sigma} \int_{t_0}^t \left( 1 + \frac{(a|\tau|+b)	\Gr{k_2} (|\tau|)}{\Gr{k_3} (|\tau|)} \right) \di s	 \leq 2 \abs{t-t_0}.
  \end{align*}
  By \cite[Lemma 7.10]{HG15reg} continuity of $\gamma$ suffices to conclude that $\gamma$ is $C^1$ and that $\Evol_{\controlledChar{\Hopf}{B}} (\eta) = \gamma$. 
  This concludes the proof.
 \end{proof}

We may now show that the group of controlled characters is even $C^0$-regular:
 
 \begin{thm}[Regularity of the group of controlled characters]							\label{thm_regularity}
  Assume that $(\Hopf, \Sigma)$ is an (RLB) Hopf algebra and $((\Hopf, \Sigma), (\Gr{n})_{n\in \N})$ a control pair.
  Then $\controlledChar{\Hopf}{B}$ is $C^0$-regular (with an analytic evolution map) and in particular regular in Milnor's sense.
 \end{thm}
 \begin{proof}
  We will assume throughout this proof that $\K=\C$. 
  The regularity of a group of real controlled characters can be deduced from the complex case by \cite[Corollary 9.10]{HG15reg} since Proposition \ref{prop: complexification} shows that the complexification of a group of controlled characters is the complex group of controlled characters.
  In Proposition \ref{prop_semiregularity} we have shown that for every $\eta\in C([0,1],\controlledInfChar{\Hopf}{B})$ there is $\gamma_\eta=\Evol(\eta)\in C^1([0,1],\InfChar{\Hopf}{B}$ such that   $\gamma_\eta' = \gamma_\eta \star \eta$. It remains to establish smoothness of
  \[
   \func{\evol}{\controlledInfChar{\Hopf}{B}}{\controlledChar{\Hopf}{B}}{\eta}{\gamma_\eta(1).}
  \]

  \paragraph{Step 1: Working in charts}	
  The $C^\omega_\K$-manifold $\controlledChar{\Hopf}{B}$ is $C^\omega_\K$-diffeomorphic to the (LB)-space $\ellInftyLimit{\Sigma}{B}$ via
  \[
   \func{\res_G}{\controlledChar{\Hopf}{B}}{\ellInftyLimit{\Sigma}{B}}{\phi}{\phi|_\Sigma},
  \]
  by Proposition \ref{prop_characters_form_submanifold}.
  Furthermore, Proposition \ref{prop_Lie_algebra_of_controlled_characters} entails that the locally convex space $\controlledInfChar{\Hopf}{B}$ is topologically isomorphic to the (LB)-space $\ellInftyLimit{\Hopf}{B}$ via
  \[
   \func{\res_\g}{\controlledInfChar{\Hopf}{B}}{\ellInftyLimit{\Sigma}{B}}{\psi}{\psi|_\Sigma.}
  \]
  Using these identifications it remains to establish smoothness of the map
  \[
   \func{\Phi}{C([0,1],\ellInftyLimit{\Sigma}{B})}{\ellInftyLimit{\Sigma}{B}}{\eta}{ \res_G( \evol(\res_\g^{-1}  \circ \eta )). }
  \]

  \paragraph{Step 2: The auxiliary maps $\Phi_{k,\ell}$ for $\ell \gg k$.}	
  Let $k\in\N$.
  In the proof of Proposition \ref{prop_semiregularity} we have seen that there is a $\ell>k$ such that
  \[
      \sup_{t\in[0,1]}\norm{ \gamma_\eta(t) }_{\ell^\infty_{\ell}} \leq 1 \quad \text{ whenever} \quad \sup_{t\in[0,1]}\norm{\eta(t)}_{\ell^\infty_k} \leq 1.
  \]
  Note that the isomorphism $\smfunc{\res_\g}{\controlledInfChar{\Hopf}{B}}{\ellInftyLimit{\Sigma}{B}}$ is isometric with respect to the $\ell^\infty_k$-norms since an infinitesimal character is zero on all words of length different from $1$ (as the Hopf algebra is connected).
  This means that we obtain a function $\smfunc{\eta}{[0,1]}{\oBallin{1}{\ellInfty{k}{\Sigma}{B}}{0}}$ via
  \[
   \func{\Phi_{k,\ell}}{\oBallin{1}{C([0,1],\ellInfty{k}{\Sigma}{B})}{0}}{\ellInfty{\ell}{\Sigma}{B}}{\eta}{\Phi(\eta)=\res_G( \evol(\res_\g^{-1}  \circ \eta )),}
  \]
  
  \paragraph{Step 3: $\Phi_{k,\ell}$ is continuous for $\ell \gg k$.}	
  Let now $k\in\N$ be fixed. We have seen that there is a number $\ell_1>k$ such that $\Phi_{k,\ell_1}$ is well-defined. 
  Let again $a,b>0$ be the constants for $\Hopf$ such that \eqref{eq: linear:est} holds. 
  Now apply \ref{axiom_H_Infinity} once more to find $\ell_2 > \ell_1$ with 
  \[
   (an+b+1)e^{an+b}\Gr{\ell_1}(n) \leq \Gr{\ell_2} (n) \quad \text{ for all }n \in \N.
  \]
  We will show that $\smfunc{\Phi_{k,\ell_2}}{\oBallin{1}{C([0,1],\ellInfty{k}{\Sigma}{B})}{0}}{\ellInfty{\ell_2}{\Sigma}{B}}$ is continuous.
  To this end, let $\epsilon>0$ and choose $\delta:=\epsilon$.
  Fix $\eta_1,\eta_2\in \oBallin{1}{C([0,1],\ellInfty{k}{\Sigma}{B})}{0}$ with
  \[
   \sup_{t\in[0,1]}\norm{\eta_1(t)-\eta_2(t)}_{\ell^\infty_k}< \delta=\epsilon.
  \]
  Consider the Lie group valued curves $\gamma_i := \Evol_{\controlledChar{\Hopf}{B}} (\res_\g^{-1}\circ \eta_i)$.
  Continuity of $\Phi_{k,\ell_2}$ holds if $\norm{\gamma_1-\gamma_2}_{\ell^\infty_{\ell_2}}\leq \epsilon$.
  Similar to the proof of Proposition \ref{prop_semiregularity} define
  \[
   g_n (t) := \sup_{\tau \in \Sigma, |\tau| \leq n} \frac{\norm{\gamma_1(t)(\tau) - \gamma_2(t)(\tau)  }}{\Gr{\ell_1} (|\tau|)} \quad \quad , t \in [0,1], \ n \in \N
  \]
  Now for a $t\in[0,1]$ we have the estimate:
  \newcommand{\bigsup}{\sup_{\tau \in \Sigma  |\tau| \leq n}}
  \newcommand{\bigsum}{\sum_{\substack{|\alpha | + |\beta| = |\tau| \\  \alpha,\beta \in \Sigma}}}
  \begin{align*}
   g_n(t)	  &	=	\bigsup \frac{1}{\Gr{\ell_1} (|\tau|)} \norm{\int_0^t \left(  \gamma_1'(s)(\tau) - \gamma_2'(s)(\tau) \right) \di s }
         	\\&	=	\bigsup \frac{1}{\Gr{\ell_1} (|\tau|)} \norm{\int_0^t \left(  (\gamma_1(s)\star\eta_1(s))(\tau) - (\gamma_1(s)\star\eta_1(s))(\tau) \right) \di s }
		\\&	= 	\bigsup \frac{1}{\Gr{\ell_1} (|\tau|)} \biggl\Vert \int_0^t\Bigl(  \eta_1 (s)(\tau) - \eta_2(s)(\tau)  +
		\\& 	\phantom{=\bigsup}  + \bigsum c_{\alpha, \beta,\tau} 
						\left(\gamma_1 (s)(\alpha) \eta_1(s) (\beta) -\gamma_2 (s)(\alpha) \eta_2(s) (\beta)\right)\Bigr)\di s \biggr\Vert \\
		\\&	\leq 	\bigsup \frac{1}{\Gr{\ell_1} (|\tau|)} \int_0^t \underbrace{\norm{   \eta_1 (s)(\tau) - \eta_2(s)(\tau)  }}_{\leq \epsilon\cdot \Gr{k}(\abs{\tau})\leq \epsilon\cdot \Gr{\ell_1}(\abs{\tau})}\di s +
		\\& 	\phantom{=\bigsup}  + \bigsum \abs{c_{\alpha, \beta,\tau} }
						\int_0^t \norm{\gamma_1 (s)(\alpha) \eta_1(s) (\beta) -\gamma_2 (s)(\alpha) \eta_2(s) (\beta) } \di s\\
		\\&	\leq	\epsilon  +  \bigsup\bigsum\abs{c_{\alpha, \beta,\tau} } \frac{1}{\Gr{\ell_1} (|\tau|)}
			  \int_0^t \bigl(	\underbrace{\norm{\gamma_1(s)(\alpha)}}_{\leq \Gr{\ell_1}(\abs{\alpha})}\underbrace{\norm{\eta_1(s)(\beta)-\eta_2(s)(\beta)}}_{\leq \epsilon\cdot \Gr{k}(\abs{\beta})\leq \epsilon\cdot \Gr{\ell_1}(\abs{\beta})} +
		\\&	\phantom{\leq	\epsilon  +  \bigsup\bigsum\abs{c_{\alpha, \beta,\tau} }
			  \int_0^t \bigl(}				      
		  + \underbrace{\norm{\gamma_1(s)(\alpha)-\gamma_2(s)(\alpha)}}_{\leq g_n(s)\cdot \Gr{\ell_1}(\abs{\alpha})}\underbrace{\norm{\eta_2(s)(\beta)}}_{\leq \Gr{k}(\abs{\beta})\leq \Gr{\ell_1}(\abs{\beta})}\bigr)\di s
		\\&	\leq \epsilon + \bigsup\bigsum\abs{c_{\alpha, \beta,\tau} } \underbrace{\frac{\Gr{\ell_1}(\abs{\alpha})\Gr{\ell_1}(\abs{\beta})}{\Gr{\ell_1}(\abs{\tau})}}_{\leq 1}
				\left(	\epsilon + \int_0^t g_n(s) \di s	\right)
		\\&	\leq \epsilon + \bigsup\underbrace{\bigsum\abs{c_{\alpha, \beta,\tau} }}_{\leq a \abs{\tau} +b} \left(	\epsilon + \int_0^t g_n(s) \di s	\right)
		\\&	\leq \epsilon( 1 + a n +b) + (an+b)  \int_0^t g_n(s) \di s.
 \end{align*}
 By Gronwall's inequality (Lemma \ref{lem_gronwall}) this implies
 \begin{align*}
  g_n(t) &\leq \epsilon( 1 + a n +b) e^{(an+b)t}
 \text{ and in particular} \\
   \sup_{\tau\in\Sigma}\frac{\norm{(\gamma_1(1)-\gamma_2(1))(\tau)}}{\Gr{\ell_2}(\abs{\tau})} 
  &\leq  \sup_{\tau\in\Sigma}\frac{g_n(1)\Gr{\ell_1}(\abs{\tau})}{\Gr{\ell_2}(\abs{\tau})} 
  \leq  \sup_{\tau\in\Sigma}\frac{\epsilon( 1 + a n +b) e^{(an+b)t}\Gr{\ell_1}(\abs{\tau})}{\Gr{\ell_2}(\abs{\tau})} 
  \leq  \epsilon
 \end{align*}
 which is what we had to show.

  \paragraph{Step 4: $\Phi_{k,\ell}$ is complex analytic.}	
 Fix $k\in\N$ and choose $\ell>k$ as in Step 3  such that $\Phi_{k,\ell}$ makes sense and is continuous.
 By Lemma \ref{lem: sl_in_full} the inclusion $\controlledChar{\Hopf}{B} \rightarrow \Char{\Hopf}{B}$
 is  a Lie group morphism.
 The group on the right is $C^0$-regular with an analytic evolution map (see \cite[Theorem 2.11]{BDS16}).
 This shows that $\Phi_{k,\ell}$ is analytic when regarded as a map into the full character group which is a submanifold of the locally convex space $B^\Hopf$. 
 Since the continuous linear point evaluations separate the points, we may apply Proposition \ref{prop: avs:weaka} to conclude that $\Phi_{k,\ell}$ is $\C$-analytic.

  \paragraph{Step 5: $\Phi$ is complex analytic.}	
  As a consequence of Step 2 and Step 4 for each $k\in\N$
 \[
  \func{\Phi_k}{\oBallin{1}{C([0,1],\ellInfty{k}{\Sigma}{B})}{0}}{\ellInftyLimit{\Sigma}{B}}{\eta}{\Phi(\eta)=\res_G( \evol(\res_\g^{-1}  \circ \eta )).}
 \]
 is complex analytic and bounded. So by \cite[Theorem A]{dahmen2010} the map $\Phi$ is complex analytic. Summing up, this shows that the evolution map is complex analytic.
 \end{proof}

 \begin{rem}\label{rem: lineargrowth}
  Observe that Theorem \ref{thm_regularity} and Proposition \ref{prop_semiregularity} can be adapted to slightly more general situations. 
  The assumption that the Hopf algebra is an (RLB) Hopf algebra can be relaxed if more information on the growth bound is known.
  Namely, instead of the estimate \eqref{eq: linear:est} it suffices to require that the elementary coproduct of the right-handed combinatorial algebra $(\Hopf, \Sigma)$ satisfies for all $k \in \N$ the estimate 
  \begin{displaymath}
    \norm{\elcopro  (\tau)}_{\ell^1} = \sum_{\substack{
                                  |\alpha | + |\beta| = |\tau| \\  \alpha , \beta \in \Sigma
                                  }} |c_{\alpha ,\beta, \tau}| \leq \log \left( \frac{\Gr{\ell} (|\tau| )}{\Gr{k}(|\tau|)}\right) \quad \tau \in \Sigma, \text{ for suitable } \ell > k.
  \end{displaymath}
 Inserting this estimate in the proofs of Proposition \ref{prop_semiregularity} and Theorem \ref{thm_regularity}, one sees that they can be carried out without any further changes. 
 If the algebra is not an (RLB) algebra, our estimate indicates that functions in the growth family need to grow super exponentially fast (i.e.\ at least like $\exp \left( \lVert\elcopro  (\tau)\rVert_{\ell^1}\right)$).
 Thus the only growth family from Proposition \ref{prop: growth:fam} leading to regular Lie groups are $\Gr{k} (n) = k^n (n!)^k$ or $\Gr{k} (n) = k^{n^2}$.
 
 Observe that the estimates in the above proofs will in general be quite conservative. 
 Hence we conjecture that with better estimates, or a refinement of the techniques used in the proofs above, one should be able to obtain regularity for all Lie groups of controlled characters.
 \end{rem}

\section{Instructive examples of controlled character groups}\label{sect: examples}

In this section we discuss (controlled) character groups for certain well known combinatorial Hopf algebras. 
Analysis of these examples usually involves involved combinatorial estimates (cf.\ e.g.\ the computations following Example \ref{ex_faadiBruno}).
Such an analysis is in general beyond the scope of the present article. 
However some perspectives for research with application to numerical analysis and control theory is provided.

\subsection*{Examples from numerical analysis I: The tame Butcher group}
In this section we discuss (controlled) character groups which are inspired by application from numerical analysis. 
Namely, we consider groups which are related to the so called Butcher group and its generalisations from Lie-Butcher theory (cf.\ e.g.\ \cite{MR2790315}).
Special emphasis will be given to the power series solutions associated to elements in this group. 
This is due to the fact that these power series solutions were the motivation to consider the tame Butcher group in \cite{BS16}.   
We begin with the easiest example in this context, the character group of the so called Connes-Kreimer Hopf algebra\footnote{This Hopf algebra is most prominently studied in the Connes-Kreimer approach to perturbative renormalisation of quantum field theories, whence the name. We refer to \cite{MR2371808} for an account.} 
 \begin{nota}
 \begin{enumerate}
  \item A \emph{rooted tree} is a connected \emph{finite} graph without cycles with a distinguished node called the \emph{root}.
  We identify rooted trees if they are graph isomorphic via a root preserving isomorphism.

  Let $\RT$ be \emph{the set of all rooted trees} and write $\RT_0 := \RT \cup \{\emptyset\}$ where $\emptyset$ denotes the empty tree.
  The \emph{order} $|\tau|$ of a tree $\tau \in \RT_0$ is its number of vertices.
  \item An \emph{ordered subtree}\footnote{The term ``ordered'' refers to that the subtree remembers from which part of the tree it was cut.} of $\tau \in \RT_0$ is a subset $s$ of all vertices of $\tau$ which satisfies
    \begin{itemize}
     \item[(i)] \ $s$ is connected by edges of the tree $\tau$,
     \item[(ii)] \ if $s$ is non-empty, then it contains the root of $\tau$.
    \end{itemize}
   The set of all ordered subtrees of $\tau$ is denoted by $\OST (\tau)$.
   Further, $s_\tau$ denotes the tree given by vertices of $s$ with root and edges induced by $\tau$.
  \item A \emph{partition} $p$ of a tree $\tau \in \RT_0$ is a subset of edges of the tree.
 We denote by $\cP (\tau)$ the set of all partitions of $\tau$ (including the empty partition).
 \end{enumerate}
  Associated to $s \in \OST (\tau)$ is a forest $\tau \setminus s$ (collection of rooted trees) obtained from $\tau$ by removing the subtree $s$ and its adjacent edges.
  Similarly, to a partition $p \in \cP (\tau)$ a forest $\tau \setminus p$ is associated as the forest that remains when the edges of $p$ are removed from the tree $\tau$.
  In either case, we let $\# \tau \setminus p$ be the number of trees in the forest.
  \item We denote by $G^\K_{TM}$ the Butcher group over $\K$ and recall that it is the set of tree maps $G^\K_{TM} = \{a \colon \RT \cup \{\emptyset\} \mid a(\emptyset)=1\}$.
        The multiplication of the Butcher group corresponds to the composition of formal power series (see below \ref{defn: Bseries}).
  \end{nota}

\begin{ex}[Controlled characters of the Connes-Kreimer algebra of rooted trees]\label{ex: CKHopf}
  Consider the algebra $\Hopf^{\K}_{CK} := \K [\cT]$ of polynomials which is generated by the trees in $\RT$.
  One defines a coproduct and an antipode on the trees as follows
  \begin{align*}
     \Delta \colon \Hopf^\K_{CK} &\rightarrow \Hopf^\K_{CK} \otimes \Hopf^\K_{CK} ,\quad  \tau \mapsto \sum_{s \in \OST (\tau)} (\tau \setminus s) \otimes s, \\
      S \colon \Hopf^\K_{CK} &\rightarrow \Hopf_{CK}^\K ,\quad \tau \mapsto \sum_{p \in \cP (\tau)} (-1)^{\# \tau \setminus p} (\tau \setminus p)
    \end{align*}
 One can show that with these structures $\Hopf^\K_{CK}$ is a $\K$-Hopf algebra which is graded and connected with respect to the number of nodes grading (see \cite[5.1]{CHV2010} for details).
 
 Since $|\OST (\tau)| \leq 2^{|\tau|}$ and $|\cP (\tau)| \leq 2^{|\tau|}$ (cf.\ e.g.\ \cite[Appendix B]{BS16}) 
 it follows from the formulae for the antipode and the coproduct that 
 $((\Hopf_{CK}^\K, \RT) , (\Gr{n})_{n \in \N})$ is a control pair for all growth families discussed in Proposition \ref{prop: growth:fam}.

 Recall from \cite[Lemma 4.8]{BDS16} that the group of all characters of the Connes-Kreimer Hopf algebra $\Hopf_{CK}^\K$ corresponds to the Butcher group from numerical analysis 
  \begin{equation}
   \Char{\Hopf_{CK}^\K}{\K} \rightarrow G^\K_{TM} , a \mapsto a|_{\RT \cup \{\emptyset\}}.
  \end{equation}
We can construct the group of controlled characters with respect to every growth family in Proposition \ref{prop: growth:fam}.
 Consider $\controlledChar{\Hopf_{CK}^\K}{\K}$ induced by the control pair $((\Hopf_{CK}^\K, \RT) , (n \mapsto k^n)_{k \in \N})$.
 Then the correspondence of the character group with the Butcher group identifies $\controlledChar{\Hopf_{CK}^\K}{\K}$ with the tame Butcher group constructed in \cite{BS16}.
  
 To see this, recall that the weights used in \cite{BS16} in the construction of the tame Butcher group are of the form $\Gr{k} (n) := 2^{kn} = (2^k)^n$.
 It is easy to see that $(\Gr{k})_{k \in \N}$ is a convex growth family and a cofinal subsequence of the convex growth family.
 Thus the inductive limits of the weighted $\ell^\infty$-spaces constructed with respect to both families coincide, as do the Lie group structures from Theorem \ref{thm: controlled Lie group} and \cite[Theorem 2.4]{BS16}.
 \end{ex}
 
We recall now how controlled characters lead to (locally) convergent power series of interest to numerical analysis.
 
 \begin{defn}[Elementary differentials and B-series]\label{defn: Bseries}
Let now $f \colon E \supseteq U \rightarrow E$ be an analytic mapping on an open subset of the normed space $(E,\norm{\cdot})$. For $\tau \in \RT$ define recursively the \emph{elementary differential} $F_f (\tau) \colon U \rightarrow E$ via
 $F_f(\bullet) (y) = f(y)$ and
  \begin{equation}\label{eq: el:diff}
   F_f(\tau)(y) \coloneq f^{(m)} (y) (F_f (\tau_1)(y), \ldots F_f (\tau_m)(y)) \quad \text{ for } \tau = [\tau_1, \ldots , \tau_m]
  \end{equation}
 where $f^{(m)}$ denotes the $m$th \Frechet derivative.
 Now we can define for an element $a \colon \RT \rightarrow \K$ in the Butcher group, $y \in U$ and $h \in \R$ a formal series
  \begin{equation}\label{Bseries}
   B_f (a,y,h) \coloneq y + \sum_{\tau \in \RT} \frac{h^{|\tau|}}{\sigma (\tau)}a(\tau) F_f (\tau)(y),
  \end{equation}
 called \emph{B-series}. The $\sigma (\tau)$ are the so called symmetry coefficients (used to normalise the series).
As $f$ is usually fixed, the dependence on $f$ will be suppressed in the notation of elementary differentials and B-series.
\end{defn}
 The B-series $B_f(a,y,h)$ corresponds to a numerical approximation of a power series solutions near $y_0$ of the ordinary differential equation 
 \begin{displaymath}(\star) \begin{cases}
              y'(t) &=f(y),\\
              y(0) &= y_0
     \end{cases}.
 \end{displaymath}
 One would like the formal B-series to converge at least locally, but B-series do not exhibit this behaviour in general. However, the B-series of a controlled character converges at least locally (cf.\ \cite[Section 6]{BS16}).

 \subsection*{Example: The Fa\`{a} di Bruno algebra} 	
 
 The aim of the present chapter is twofold. First of all, we recall the construction of the so called Fa\`{a} di Bruno Hopf algebra.
 This Hopf algebra encodes the combinatorial structure of the composition of formal diffeomorphisms on vector spaces.
 As is well known, this composition can be described via the famous Fa\`{a} di Bruno formula (cf.\ e.g.\ \cite{MR3381489}), whence the name of the Hopf algebra.
 In a second step we consider another realisation of this Hopf algebra as a combinatorial Hopf algebra. 
 This example emphasises the importance of the explicit choice of basis for the combinatorial Hopf algebra as a change of base yields different behaviour of the controlled characters.
 
 \begin{rem} The dependence on the basis should be seen as a feature of our approach and not as a defect: In contrast to the construction for the full character group which is insensitive to the grading \cite[3.10 Remark]{BDSOverview}, the construction depends on the grading. Thus the group of controlled characters exhibits new features which are desirable in the renormalisation of quantum field theories (S.\ Paycha, private communications.).
 \end{rem} 
 
 We concentrate here on the classical commutative Fa\`{a} di Bruno algebra.\footnote{There are many Hopf algebras (and bialgebras) which are of ``Fa\`{a} di Bruno type''. For another example, we refer to Example \ref{ex: fdB:Chen} below.} 
 Our exposition follows mainly \cite[Section 3.2.3]{MR3254718}, but we need the explicit formula for the antipode as recorded in \cite[Theorem 2.14]{MR2200854}. (Observe that our choice of variables is compatible with the one in loc.cit. by \cite[Remark 2.12]{MR2200854}.)
 
 \begin{ex}[The commutative Fa\`{a} di Bruno algebra]							\label{ex_faadiBruno}
  The \emph{Fa\`{a} di Bruno algebra} is the combinatorial right-handed Hopf algebra $\Hopf_{FdB}$ given by the following data:
  \begin{itemize}
   \item Fix a graded index set $\Sigma :=\set{a_1,a_2,\ldots}$ with $\abs{a_n}:=n$. 
   Conceptually, the variables correspond to the coordinate functions of formal diffeomorphisms, i.e.\ $a_n (\phi) = \tfrac{1}{(n+1)!} \tfrac{\di^{n+1}}{\di x^{n+1}} \phi (0)$. 
   \item As an algebra $\Hopf_{FdB}$ is the commutative polynomial algebra $\C[\Sigma]$ with the connected grading induced by $\Sigma$.
    The counit $\smfunc{\epsilon}{\Hopf_{FdB}}{\C}$ is projection onto the $0$th component (as in every connected Hopf algebra).
   \item The coproduct $\smfunc{\Delta}{\Hopf_{FdB}}{\Hopf_{FdB}\otimes \Hopf_{FdB}}$ is given by the formula
	\begin{align}
	 \Delta (a_n) &:= \sum_{r=0}^n 
			  \sum_{\substack{\beta_1 + 2\beta_2 + \cdots + n\beta_n = n-r \label{eq: co1fdb:l1} \\ 
					  \beta_0 + \beta_1 + \cdots + \beta_n = r+1}}
			   \frac{(r+1)!}{\beta_0! \beta_1! \cdots \beta_n!} a_r \otimes a_1^{\beta_1}\cdots a_n^{\beta_n}.\\
			   &= \sum_{r=0}^n \frac{(r+1)!}{(n+1)!} B_{n+1,r+1} (1!a_0,2!a_1,\ldots , (n-r+1)!a_{n-r+1})  .\label{eq: co1fdb:l2}			   
	\end{align}
	Here the $B_{n,r}$ are the partial Bell polynomials (cf.\ \cite[Section 3.3]{MR0460128}). Further, we identify $a_0 :=1$ as this yields the Fa\`{a} di Bruno bialgebra. 
   \item The antipode $\smfunc{S}{\Hopf_{FdB}}{\Hopf_{FdB}}$ is given by the explicit formula:
	\begin{align}\label{eq: anti1:fdB}
	 S(a_n) := -a_n - \sum_{r=1}^{n-1} (-1)^r
			   \sum_{\substack{n_1+\cdots + n_r+n_{r+1} = n \\
					   n_1, \ldots, n_{r+1} >0 }}
			    \lambda(n_1,\ldots, n_r)\ a_{n_1} \cdots a_{n_r} a_{n_{r+1}}, \\
	\text{where } \label{eq: anticoeff:fdb}
         \lambda(n_1,\ldots, n_r)=\sum_{\substack{ m_1 + \cdots + m_r = r		\\
							m_1 + \cdots + m_h \geq h	\\
							h=1, \ldots , r-1}}
			    \binom{n_1 + 1}{m_1}\cdots \binom{n_r+1}{m_r}.
        \end{align}
        This formula was derived in \cite{MR2200854} for a \emph{non-commutative} Fa\`{a} di Bruno algebra. Its abelianisation yields an explicit formula for the commutative case.
    \end{itemize} \end{ex}

 \begin{prop}\label{prop controlled:fdb}
 Together with any growth family from Proposition \ref{prop: growth:fam} the combinatorial Hopf algebra $(\Hopf_{FdB}, \Sigma)$ forms  a control pair.
 \end{prop}
 
 \begin{proof}
  Once the estimates to apply Lemma \ref{lem_continuity_of_linear_maps} are established, $\ell^1$-continuity follows.
  Let us first consider the coproduct and recall from \cite[p.\ 135 Theorem B]{MR0460128} $$B_{n+1,r+1} (1, 2!, 3!,\ldots (n-r+1)!) =\binom{n}{r} \frac{(n+1)!}{(r+1)!}.$$
  Fix $k \in \N$ and use the formula \eqref{eq: co1fdb:l2} to obtain for $n \in \N$ the following
    \begin{align*}
     \norm{\Delta (a_n)}_{\ell^1_k} &= \norm{\sum_{r=0}^n \frac{(r+1)!}{(n+1)!} B_{n+1,r+1} (1!a_0,2!a_1,\ldots , (n-r+1)!a_{n-r+1})}_{\ell^1_k} \\
				    &\leq \sum_{r=1}^{n} B_{n+1,r+1} (1!,2!,\ldots,(n-r+1)!) \Gr{k} (|a_n|)\\
				    &= \Gr{k}(n)\sum_{r=1}^n \frac{(r+1)!}{(n+1)!}\binom{n}{r} \frac{(n+1)!}{(r+1)!} \leq \Gr{k}(n) 2^n.
    \end{align*}
 To pass to the second line we used $\prod_{i=0}^{n-k+1} \Gr{k} (|a_{i}|) \leq \Gr{k} (|a_n|)$ (cf.\ the explicit \eqref{eq: co1fdb:l1}). 
   
  We now turn to the antipode and use the formula \eqref{eq: anti1:fdB}. 
  Note that we immediately have to take an estimate as we are not working with the abelianised version of the formula.
  Again fix $k \in \N$ and let $n \in  \N$. Then we compute 
  \begin{align*}
   \norm{S(a_n)}_{\ell^1_k} &\leq \Gr{k}(|a_n|) + \sum_{r=1}^{n-1} \sum_{\substack{n_1+\cdots + n_r+n_{r+1} = n \\
					   n_1, \ldots, n_{r+1} >0 }} \hspace{-3em}
			    \lambda(n_1,\ldots, n_r)\Gr{k}(|a_{n_1}|) \cdots \Gr{k}(|a_{n_{r+1}}|)\\
			    &\leq \Gr{k}(|a_n|)\left(1+ \sum_{r=1}^{n-1} \sum_{\substack{n_1+\cdots + n_r+n_{r+1} = n \\
					   n_1, \ldots, n_{r+1} >0 }} \hspace{-3em}
			    \lambda(n_1,\ldots, n_r) \right) .
  \end{align*}
  Recall from \eqref{eq: anticoeff:fdb} that the coefficients $\lambda(n_1,\ldots, n_r)$ in the above formula consist of a sum of products of binomial coefficients.
  Taking a very rough estimate (observe the condition in the second sum of \eqref{eq: anticoeff:fdb}!), ever summand in this sum is certainly smaller than $2^{2n} = 4^n$.
  To complete the estimate on $\lambda(n_1,\ldots, n_r)$, recall from \cite[2.4]{MR2200854} that this number equals the $r$th Catalan number $C_r := \frac{1}{r+1} \binom{2r}{r}$.
  Since $r < n$ we obtain 
  \begin{displaymath}
   \lambda(n_1,\ldots, n_r) \leq C_r 4^n \leq 2^{2n} 4^n = 16^n.
  \end{displaymath}
 Inserting this into the above estimate, we obtain 
  \begin{align*}
    \norm{S(a_n)}_{\ell^1_k} &\leq \Gr{k}(|a_n|) \left(1 +  \sum_{r=1}^{n-1} \sum_{\substack{n_1+\cdots + n_r+n_{r+1} = n \\
					   n_1, \ldots, n_{r+1} >0 }} 16^n \right) \leq  \Gr{k}(|X_n|) 16^n 2^{n-1}
  \end{align*}
 For the last estimate we used that the summands do not depend on the summation condition, whence only the number of summands is important.
 As this number is the number of compositions of the number $n$ into $r$ smaller natural numbers, it is smaller than $2^{n-1}$ (cf.\ \cite[Ex.\ 23 on p.\ 123]{MR0460128}).
 Since all growth families grow at least exponentially in $n$, the continuity of the antipode follows (e.g.\ for $\ell > 32k$).
 \end{proof}

 Though the Fa\`{a} di Bruno algebra is a right-handed Hopf algebra, it is unfortunately not an (RLB) Hopf algebra as the next computation with \eqref{eq: co1fdb:l1} shows for $n\in \N$: 
 \begin{align*}
  \norm{\elcopro (a_n)}_{\ell^1} &= \norm{ \sum_{r=1}^{n-1} 
			  \sum_{\substack{\beta_1 + 2\beta_2 + \cdots + n\beta_n = n-r \\ 
					  \beta_1 + \cdots + \beta_n = 1, \beta_i \in \{0,1\}}}
			   \frac{(r+1)!}{r! \beta_1! \cdots \beta_n!} a_r \otimes a_1^{\beta_1}\cdots a_n^{\beta_n}}_{\ell^1} \\
			   &= \norm{\sum_{r=1}^{n-1} (r+1) a_r \otimes a_{n-r}}_{\ell^1} = \sum_{r=1}^{n-1} (r+1) = \frac{n(n+1)}{2} -1 
 \end{align*}
 Hence, the Fa\`{a} di Bruno Hopf algebra is not an (RLB) Hopf algebra and we can not use Theorem \ref{thm_regularity} to establish regularity of its controlled character groups.
 However, as $\exp (\frac{n(n+1)}{2} -1 ) < k^{n^2}$ for large enough $k$, at least the Lie group of controlled characters with respect to the growth family $\Gr{k}(n) = k^{n^2}$ is $C^0$-regular by Remark \ref{rem: lineargrowth}.

  We now realise the Fa\`{a} di Bruno Hopf algebra as a different combinatorial Hopf algebra turn to realise another choice of variables for the Fa\`{a} di Bruno algebra, i.e.\ we realise .
  Note that the only changes in the following example will be a change of base of the Hopf algebra. All other structure maps remain the same (though the formulae have to be expressed in the new basis).
  
   \begin{ex}[Fa\`{a} di Bruno algebra II: Another combinatorial structure]\label{ex_fdB2}
  We construct the Fa\`{a} di Bruno algebra as in Example \ref{ex_faadiBruno}. However, we scale the variables
  \begin{displaymath} 
     \Sigma' := \{X_n \coloneq (n+1)! a_n \mid n \in \N\}
    \end{displaymath} 
    and consider $\Hopf_{FdB}$ as a combinatorial Hopf algebra with respect to $\Sigma'$.
  One introduces the coordinate transformation to the $X_n$ because the coproduct \eqref{eq: co1fdb:l2} changes to 
  \begin{equation}\label{eq: co2:fdb}
   \Delta (X_n) = \sum_{k=0}^n X_k \otimes B_{n+1,k+1} (X_0,X_1, X_2, \ldots , X_{n-k+1}) \quad n \in \N
  \end{equation}
 where $B_{n+1,k+1}$ is the partial Bell polynomial and $X_0 = 1$.
 For the antipode an easy computation shows that \eqref{eq: anti1:fdB} yields the following in the new variables. 
 \begin{equation}\label{eq: anti2fdb}
  S(X_n) = -X_n - \sum_{r=1}^{n-1} (-1)^r \hspace{-1em} \sum_{\substack{n_1+\cdots + n_r+n_{r+1} = n \\
					   n_1, \ldots, n_{r+1} >0 }}\hspace{-2em}
			    \lambda(n_1,\ldots, n_r) \frac{(n+1)!X_{n_1} \cdots X_{n_r} X_{n_{r+1}} }{(n_1+1)! \cdots ((n_{r+1}+1)!)} 
 \end{equation}
 We will now see that $((\Hopf_{FdB}, \Sigma'),(\Gr{k})_{k\in \N})$ is a control pair for the growth families from Proposition \ref{prop: growth:fam} if the growth family is given by $\Gr{k}(n) = k^n (n!)^k$.
 
 To this end recall from \cite[p.\ 135 Theorem B]{MR0460128} that the partial Bell polynomial $B_{n,k} (1,\ldots,1)$ yields the Stirling number of the second kind.
 The Stirling numbers of the second kind add up to the Bell numbers (see \cite[Section 5.4]{MR0460128} , for which asymptotic growth bounds are known (cf.\ \cite{MR2792580}).
 We derive for $n \in \N$ the bound 
  \begin{align*}
   \norm{\Delta (X_n)}_{\ell^1_k} & < \Gr{k}(|X_n|) \left( \frac{0.792(n+1)}{\log (n+1)}\right)^{n+1} \leq \Gr{k}(|X_n|) (n+1)^{n+1}.
  \end{align*}
  Recall the estimate $e\left(\tfrac{n}{e}\right)^n<n!$, whence $(n+1)^{n+1} < (n+1)! e^{n+1} = e(n+1)(n!)e^n \leq e(2e)^n(n!)$.
  Now as $\Gr{k}(n) = k^n (n!)^k$, we see that for $\ell > 2ek$ we obtain the estimate 
    \begin{displaymath}
     \norm{\Delta (X_n)}_{\ell^1_k} < e \Gr{\ell} (|X_n|) \quad \forall n\in \N.
    \end{displaymath}
    Similar to the estimate obtained in Proposition \ref{prop controlled:fdb} for the antipode one establishes an $\ell^1_k$ bound for the antipode from \eqref{eq: anti2fdb}. 
    Thus $((\Hopf_{FdB}, \Sigma'),(n\mapsto k^n(n!)^k)_{k\in \N})$ is a control pair. 
    However, as the above estimates show, the general growth behaviour of the coproduct and the antipode do not allow one to form e.g.\ a group of exponentially bounded characters.
    Note further that also $(\Hopf_{FdB}, \Sigma')$ is not an (RLB)-Hopf algebra (an estimate of the elementary coproduct shows that it grows as $2^{|X_n|}$).
 \end{ex} 

 \subsection*{Example from numerical analysis II: Partitioned methods}
 We will now construct controlled characters in the context of partitioned Butcher series (also called P-series, cf.\ \cite{CHV2010,MR1666537}). 
 To this end, we consider coloured trees. 
 
 \begin{defn}
  Consider a set of colours $\colours = \{ 1, 2, \ldots , N\}$ for $N \in \N \cup \{\infty\}$. 
  A coloured rooted tree $\tau$ is a rooted tree $\tau \in \RT$, where each node has been marked with one of the symbols (colours) $i \in \colours$.
  We let $\RT_{\colours}$ be the \emph{set of coloured rooted trees} and identify $\colours$ with the set of coloured one-node trees in $\RT_{\colours}$.
  A subtree $\sigma$ of the coloured rooted tree $\tau$ is a coloured tree by endowing it with the colouring inherited from $\tau$. 
 \end{defn}

 \begin{setup}[Coloured Connes-Kreimer Hopf algebra]\label{ex: colCKHopf}
 Let $\colours$ be a set of colours and consider the algebra $\Hopf_{CK, \colours}^\K : = \K [\RT_{\colours}]$. 
 This algebra is a graded and connected Hopf algebra with respect to the coproduct, antipode and the number of nodes grading from Example \ref{ex: CKHopf}.
 We call this Hopf algebra the \emph{$\colours$-coloured Connes-Kreimer Hopf algebra}.

 Every growth family from Proposition \ref{prop: growth:fam}, $((\Hopf_{CK}^\K, \RT) , (\Gr{n})_{n \in \N})$ is a control pair (as in Example \ref{ex: CKHopf}). 
 Hence Theorem \ref{thm: controlled Lie group} allows us to construct controlled character groups. Again, one is interested in control pairs which ensure (local) convergence.
 \end{setup}

 In the discussion we restrict to the important special case $\colours = \{ 1,2\}$ (the general case is similar).
 To ease the discussion, associate $1$ with ''white'' and $2$ with ''black''.
 
 \begin{setup}[Partitioned systems of ODE's and P-series]\label{setup: PSeries}
  Let $d \in \N$ and fix analytic functions $f,g \colon \R^{2d} \rightarrow \R^{d}$. 
  We consider the partitioned ordinary differential equation 
    \begin{equation}\label{eq: part:ODE}
     \begin{cases}
      \dot{p} = f(p,q), &  p(0)=p_0\\
      \dot{q} = g(p,q), &  q(0)=q_0 
     \end{cases}
    \end{equation}
  Equations of this type appear naturally for example in the treatment of mechanical systems (distinguishing positions and velocities) ot in rewriting a second order differential equation as a system of equations. 
  In these instances, numerical schemes should respect the special structure of \eqref{eq: part:ODE}.
  This leads to Runge-Kutta-Nystr\"{o}m methods and additive Runge-Kutta methods (cf.\ \cite[Introduction]{MR1666537}, \cite[Section III.2]{MR2221614}).    
  
  Numerical (power-series) solutions to \eqref{eq: part:ODE} can be constructed similarly to the B-series already discussed. 
  To this end one augments the definition of elementary differentials \eqref{eq: el:diff} to encompass coloured trees. 
  Whenever the node in the tree is white, we insert the partial derivative of $f$ with respect to $p$ and if the node is black we insert the differential of $g$ with respect to $q$.
  Consider the following explicit example
  \begin{align*}
  \tau =  \begin{tikzpicture}[dtree, scale=1.5]
                     \node[dtree white node] {}
                     child { node[dtree black node] {}
                     child { node[dtree black node] {}
                     }
                     child { node[dtree white node] {}
                     }}
                     child {node[dtree black node] {}}
                     child { node[dtree black node] {}}
                     ;
                   \end{tikzpicture}  \rightsquigarrow F(\tau)(p,q) = d_1^3 f (p,q; d_2^2 g (p,q; g(p,q) , f(p,q)) , g(p,q),g(p,q)) 
  \end{align*}
 Using the elementary differentials with respect to coloured trees, one defines the P-series\footnote{This presentation of the P-series follows \cite[III. Definition 2.1]{MR2221614}. 
 Note that loc.cit.\ defines for every colour in $\colours$ a different empty tree. This leads to a non-connected Hopf algebra. We avoid this by allowing just one empty tree and refer to \cite{MR1666537} for the development of the theory.} of a map $a \colon \RT_{\colours} \cup \{\emptyset\} \rightarrow \K$ as 
    \begin{align*}
     P_{(f,g)} (a,h,(p,q)) = \def\arraystretch{2}\begin{pmatrix}
                              a(\emptyset)p + \displaystyle\sum_{\substack{\tau \in \RT_{\colours}\\ \text{root of }\tau \text{ is white}}} \frac{h^{|\tau|}}{\sigma (\tau)} a(\tau) F(\tau) (p,q)  \\
                              a(\emptyset)q + \displaystyle\sum_{\substack{\tau \in \RT_{\colours}\\ \text{root of }\tau \text{ is black}}} \frac{h^{|\tau|}}{\sigma (\tau)} a(\tau) F(\tau) (p,q) 
                             \end{pmatrix}.
    \end{align*}
 Arguing as in the case of the tame Butcher group in \cite{BS16}, one can now obtain interesting groups of controlled characters for the coloured Connes-Kreimer Hopf algebras.
 In the special case of two colours, the controlled characters induced by the control pair $(\Hopf_{CK, \colours}^\K,  (n \mapsto k^n)_{k\in \N})$ correspond to a group of (locally) convergent P-series.
 \end{setup}

 \subsection*{Outlook: Non-commutative case, Lie-Butcher theory}
 
 In numerical analysis the Butcher group, B-series and P-series as discussed above are geared towards understanding numerical integration schemes for differential equations on euclidean space.
 There are several generalisations of this concept which are of interest to numerical analysis.
 The other regimes of interest often involve distinct features of non-commutativity. 
 For example, if one wants to treat (autonomous) differential equations evolving on manifolds one has to deal with the non-commutativity of the differentials occurring in this setting.
 This lead to the development and study of Lie-Butcher series and their associated (non-commutative) Hopf algebra. 
 On one hand, one can model these integrators using the shuffle algebra (autonomous case), cf.\ \cite{MR2790315}. 
 Note that the shuffle algebra and its characters appear in the treatment of word series with application to numerical integration as in \cite{MR3648103,MR3485151}.
 On the other hand, one can consider a Fa\`{a}-di Bruno type algebra \cite{MR2790315} to model the non-autonomous case.
 We will now recall from \cite{MR3648103} and \cite{MR2790315} the constructions of power series solutions associated to characters of the shuffle algebra.

 \paragraph{Power series and word series, \cite{MR3648103}} 
 Consider on $\R^d$ the initial value problem 
    \begin{equation}\label{eq: Wseries}
     \begin{cases}
      \frac{\di}{\di t} x(t)  &= \sum_{a \in \mathcal{A}} \lambda_a (t) f_a (x)\\
      x(0)&=x_0      
     \end{cases}
    \end{equation}
 where $\mathcal{A}$ is a finite or countably infinite alphabet with the grading $\rho \equiv 1$, the $\lambda_a$ are scalar-valued functions and the $f_a$ are $\R^d$-valued.  
 For a word $w \in \mathcal{A}^*$ one constructs recursively the \emph{word-basis function} $f_w$ by combining partial derivatives of the functions $f_a$ using the product rule.
 To illustrate this, we recall the following example from \cite[Remark 2]{MR3648103} (derivatives in the following are \Frechet derivatives): 
 \begin{align*}
  f_{ba} (x) &= f_a^{'} (x)(f_b(x)) \\
  f_{cba} (x) &= f_{ba}^{'} (x)f_c(x) = f^{''} f_a (x) (f_b (x),f_c(x)) + f_a^{'} (x)(f_b^{'}(x)(f_c(x))
 \end{align*}
 Thus the word-basis function $f_w$ is a linear combination of elementary differentials (as in \ref{setup: PSeries}) in the functions $f_a, a\in \mathcal{A}$ for some trees.\
 For a character $\delta \in \Char{\text{Sh} (\mathcal{A},\rho)}{\R}$ one defines now the formal series\footnote{Which becomes a power series by introducing a step size parameter $\varepsilon >0$, cf.\ \cite[Remark 3]{MR3648103}.} 
  \begin{equation}\label{pseries}
   W_\delta (x) = \sum_{w \in \mathcal{A}^*} \delta (w) f_w (x)
  \end{equation}
 A similar construction is outlined in  \cite{MR2790315} to construct solutions to (autonomous) differential equations on a manifold.
 To treat time-dependent differential equations on a manifold, one introduces a generalised Connes-Kreimer Hopf algebra (see \cite{MR2407032}, \cite[Section 4.3.3]{MR2790315}).
 Once the algebra is defined, one obtains power series from its characters through a formula which is similar to \eqref{pseries}.
 
 \begin{setup}[Open problems]
  We refrain from discussing the details here as it would involve a combinatorial analysis.
 However, let us outline the main steps and problems:
 \begin{enumerate}
  \item Establish control pairs for the shuffle Hopf algebra (using the basis of Lyndon words) and the generalised Connes-Kreimer algebra (here obtain a basis first).
  \item Study convergence of power series induced by the controlled characters.
 \end{enumerate}

 Similar to \cite{BS16} one could hope that exponentially bounded controlled characters with Cauchy-type estimates imply local convergence.
 However, the construction of the power series \eqref{pseries} differs fundamentally from the construction of the B-series \eqref{Bseries}.
 Contrary to B-series, the power series associated to the shuffle algebra do not include a normalisation (needed to establish local convergence via \cite[Proposition 1.8]{BS16}).
 Thus exponentially growing characters might not lead to (locally) convergent series. 
 \end{setup}
 
\subsection*{Outlook: A Fa\`{a} di Bruno algebra related to Chen-Fliess series} 
 In this subsection we discuss a Hopf algebra closely connected to Chen-Fliess series from control theory.
 Our exposition follows \cite{MR2849486}.
 In control theory, one is interested in integral operators called Fliess operators and their generating series.

\begin{setup}[Generating series of Chen-Fliess operators]\label{setup: genseries}
Fix a finite, non-empty alphabet $X = \{ x_0 , \ldots, x_m\}$ and denote by $\R \langle \langle X\rangle \rangle$ the set of mappings $c \colon X^* \rightarrow \R$.
One interprets elements in  $\R \langle \langle X\rangle \rangle$ as formal series via $c= \sum_{\eta \in X^*} (c,\eta)\eta$ and identifies these series as the generating series of the Fliess operators.
Now adjoin a new symbol $\delta$ (formally the generating series of the identity) and set $c_\delta = \delta + c$ for $c\in \R\langle \langle X\rangle \rangle$.
\end{setup}

\begin{ex}[A Fa\`{a} di Bruno type Hopf algebra \cite{MR2849486}]\label{ex: fdB:Chen} 
We use the notation as in \ref{setup: genseries}.
For $\eta \in X^* \cup \{\delta\}$ define the coordinate functions 
  \begin{displaymath}
   a_\eta \colon \R \langle \langle X\rangle \rangle \rightarrow \R, c \mapsto (c,\eta) \text{ and } a_\delta \equiv 1.
  \end{displaymath}
 Define the commutative $\R$-algebra of polynomials $A = \R[a_\eta \colon \eta \in X \cup \delta]$, where the product is defined as $a_\eta a_\gamma (c_\delta) = a_\eta (c_\delta) a_\gamma (c_\delta)$.
 Then one can construct on $A$ the structure of a graded and connected Hopf algebra such that the character group corresponds to the group of Fliess operators.
 Note that we are deliberately vague here and refer to \cite{MR2849486,MR3278760} and \cite{DuffautEspinosa2016609} for an explicit description (even of a combinatorial basis). Hence $A$ is a right-handed combinatorial Hopf algebra.
\end{ex}
 
 The group $\Char{A}{\R}$ group was considered (formally) in control theory ``as an infinite-dimensional Lie group'' (cf.\ \cite[Section 5]{MR3628306}). It's Lie group structure should be the one from \cite[Theorem A]{BDS16}). 
 Now characters in $\Char{A}{\R}$ which satisfy 
 \begin{equation}\label{eq: gb:CHen}
|\phi (\eta)| < Ck^{|\eta|} |\eta|! \text{ for some } C,k >0 \text{ and all }\eta \in X^*.
\end{equation}
 correspond to locally convergent series as explained in \cite[p.\ 442]{MR2849486}.
 It was shown in \cite[Theorem 7]{MR2849486} that these bounded characters form a subgroup of the character group $\Char{A}{\R}$. 
 Hence, a detailed analysis should show that the Hopf algebra together with the growth bounds \eqref{eq: gb:CHen} form a control pair. 
 Thus we expect to obtain a Lie group of controlled characters by applying Theorem \ref{thm: controlled Lie group} which is of interest in control theory.
 Details of this construction are beyond the scope of the present paper, as they require a detailed understanding of the results in \cite{MR2849486,MR3278760,DuffautEspinosa2016609}.
 We hope to provide details and a discussion of possible applications in future work.
 
 \appendix
 
 \section{Calculus in locally convex spaces}\label{App: lcvx:diff}
 
 
\begin{defn}
 Let $r \in \N_0 \cup \set{\infty}$ and $E$, $F$ locally convex $\K$-vector spaces and $U \subseteq E$ open.
 We say a map $f \colon U \rightarrow F$ is a $C^r_\K$-map if it is continuous and the iterated directional derivatives 
  \begin{displaymath}
   d^kf (x,y_1,\ldots,y_k) \coloneq (D_{y_k} \cdots D_{y_1} f) (x)
  \end{displaymath}
 exist for all $k \in \N_{0}$ with $k \leq r$ and $y_1,\ldots,y_k \in E$ and $x \in U$,
 and the mappings $d^kf \colon U \times E^k \rightarrow F$ so obtained are continuous. 
 If $f$ is $C^\infty_\R$, we say that $f$ is \emph{smooth}.
 If $f$ is $C^\infty_\C$ we say that $f$ is  \emph{holomorphic} or \emph{complex analytic}\footnote{Recall from \cite[Proposition 1.1.16]{dahmen2011}
	    that $C^\infty_\C$ functions are locally given by series of continuous homogeneous polynomials (cf.\ \cite{BS71a,BS71b}).
	    This justifies our abuse of notation.}
 and that $f$ is of class $C^\omega_\C$.     \label{defn: analyt}
\end{defn}


\begin{defn}							\label{defn: real_analytic}							
 Let $E$, $F$ be real locally convex spaces and $f \colon U \rightarrow F$ defined on an open subset $U$. Denote the complexification of $E$ and $F$ by $E_\C$ (by $F_\C$ resp.).
 We call $f$ \emph{real analytic} (or $C^\omega_\R$) if $f$ extends to a $C^\infty_\C$-map $\tilde{f}\colon \tilde{U} \rightarrow F_\C$ on an open neighbourhood $\tilde{U}$ of $U$ in the complexification $E_\C$.
\end{defn}

For $r \in \N_0 \cup \set{\infty, \omega}$, being of class $C^r_\K$ is a local condition, 
i.e.\ if $f|_{U_\alpha}$ is $C^r_\K$ for every member of an open cover $(U_\alpha)_{\alpha}$, 
then $f$  is $C^r_\K$. (see \cite[pp. 51-52]{MR1911979} for the case of $C^\omega_\R$, the other cases are clear by definition.)  
In addition, the composition of $C^r_\K$-maps (if possible) is again a $C^r_\K$-map (cf. \cite[Propositions 2.7 and 2.9]{MR1911979}). 
 
\begin{setup}[$C^r_\K$-Manifolds and $C^r_\K$-mappings between them]
 For $r \in \N_0 \cup \set{\infty, \omega}$, manifolds modelled on a fixed locally convex space can be defined as usual. 
 The model space of a manifold and the manifold as a topological space will always be assumed to be Hausdorff. 
 However, we will neither assume second countability nor paracompactness.
 Direct products of locally convex manifolds, tangent spaces and tangent bundles as well as $C^r_\K$-maps between manifolds may be defined as in the finite-dimensional setting.
For $C^r_\K$-manifolds $M,N$ we use the notation $C^r_\K(M,N)$ for the set of all $C^r_\K$-maps from $M$ to $N$.
Furthermore, for $s \in \{\infty,\omega\}$, we define \emph{locally convex $C^s_\K$-Lie groups} as groups with a $C^s_\K$-manifold structure turning the group operations into $C^s_\K$-maps.
\end{setup}


\begin{prop}\label{prop: avs:weaka}
 Let $E,F$ be complete locally convex spaces over $\C$ and $U \subseteq E$ be an open subset. 
 A continuous $f\colon U \rightarrow F$ is complex analytic if there exists a separating\footnote{``Separating'' means that for every $v \in F$ there is $\lambda \in \Lambda$ with $\lambda (v) \neq 0$.} family $\Lambda$ of continuous linear maps $\lambda \colon F \rightarrow \C$ such that $\lambda \circ f$ is complex analytic.
\end{prop}

\begin{proof}
 It is well known (cf.\ \cite[Lemma 1.1.15]{dahmen2011}) that $f$ is complex analytic if and only if $f$ is continuous and Gateaux-analytic at every point of its domain. 
 By definition $f$ is Gateaux analytic if for every $x \in U$ and $v \in E$ there exists $\varepsilon >0$ such that  
  \begin{displaymath}
   f_{x,v} \colon B_{\varepsilon}^\C (0) \rightarrow F, \quad z \mapsto f(x+zv)
  \end{displaymath}
 is analytic. However, our assumption shows that for every such pair $(x,v)$ and $\lambda \in \Lambda$ the mapping $\lambda \circ f_{x,v}$ is complex analytic.
 As $\Lambda$ separates the points and $F$ is complete \cite[Theorem 1]{MR2040581} shows that $f_{x,v}$ is complex analytic, whence $f$ is so.
\end{proof} 


\section{Auxiliary results for Chapter 1}\label{app: aux:func}

\begin{lem}
 Let $\IndwAbs$ be a graded index set and fix a growth family $\left(\Gr{k}\right)_{k\in\N}$.

 For each $k\in\N$, the vector space $\ellOne{k+1}{\Ind}$ is a subspace of $\ellOne{k}{\Ind}$ and the inclusion
		\[
		 I_k \colon \nnfunc{\ellOne{k+1}{\Ind}}{\ellOne{k}{\Ind}}{\sum_\tau c_\tau \cdot \tau}{\sum_\tau c_\tau \cdot \tau}
		\]
		is continuous linear with operator norm at most $1$.
 \label{lem_canonical_maps}
 \end{lem}
 \begin{proof} Let $\sum_\tau c_\tau \cdot \tau\in\ellOne{k+1}{\Ind}$ be given. Then
  \begin{align*}
   \ellOneNorm{\sum_\tau c_\tau \cdot \tau}{k}
		   =					\sum_\tau \abs{c_\tau} \Gr{k}(\abs{\tau})
		 \stackrel{\textup{\ref{axiom_H_monotonic_in_k}}}{\leq}	\sum_\tau \abs{c_\tau} \Gr{k+1}(\abs{\tau})
		    =					\ellOneNorm{\sum_\tau c_\tau \cdot \tau}{k+1}.			\qedhere
  \end{align*}
 \end{proof}

 \begin{lem}\label{lem: k:isom:iso}
 Fix a graded index set $\Ind$, a growth family $\left(\Gr{k}\right)_{k\in\N}$ and a Banach space $B$.
 Let $k \in \N$ and endow $\K^{(\Ind)}$ with the $\ell^1_k$-norm and the continuous linear maps $\textup{Hom}_{\K} (\K^{(\Ind)},B)$ with the operator norm.
 Then there is an isometric isomorphism 
  \begin{displaymath}
   R \colon \textup{Hom}_{\K} (\K^{(\Ind)},B) \rightarrow \ellInfty{k}{\Ind}{B} , f \mapsto f|_{\Ind}.
  \end{displaymath}
 \end{lem}

 \begin{proof}
  Denote by $\opnorm{f}$ the operator norm of $f \in \textup{Hom}_{\K} (\K^{(\Ind)},B)$. For $\tau \in \Ind$ we derive 
    \begin{displaymath}
     \norm{R(f)(\tau)}_{B} = \norm{f(\tau)}_{B} \leq \opnorm{f}\norm{\tau}_{\ell^1_k} = \opnorm{f} \omega_k (|\tau|)
    \end{displaymath}
 Dividing by $\omega_k (|\tau|)$ and passing to the supremum over $\tau \in \Ind$ in the above inequality, we see that $\norm{R(f)}_{\ell^\infty_k} = \norm{f|_{\Ind}}_{\ell^\infty_k} \leq \opnorm{f}$.
 Conversely, for $x = \sum_{\tau \in \Ind} c_x \tau \in \K^{(\Ind)}$, 
  \begin{align*}
   \norm{f(x)}_B &\leq \sum_{\tau \in \Ind} |c_x|\norm{f(\tau)}_B = \sum_{\tau \in \Ind} |c_x|\frac{\norm{f(\tau)}_B}{\omega_k (|\tau|)} \omega_k (|\tau|)\\
		 &\leq \norm{x}_{\ell^1_k} \norm{f|_{\Ind}}_{\ell^\infty_k} = \norm{x}_{\ell^1_k} \norm{R(f)}_{\ell^\infty_k}.
  \end{align*}
 Dividing by $\norm{x}_{\ell^1_k}$ and passing to the supremum over $x$, we see $\opnorm{f} \leq \norm{R(f)}_{\ell^\infty_k}$. 
 \end{proof}

We will now establish Lemma \ref{lem_continuity_of_linear_maps} using a result which should be a standard fact from functional analysis. Unfortunately, we were unable to track it down.
 
 \begin{lem}[Factorisation Lemma]											\label{lem_factorisation}
  Let $E:=\projlim{}{E_k}$ be the projective limit of a sequence of Banach spaces $(E_1\leftarrow E_2 \leftarrow \cdots)$ 
  such that the projections $\smfunc{\pi_k}{E}{E_k}$ have dense image.\footnote{This is equivalent to the assumption that the bonding maps have dense image.}
  Let $\smfunc{\phi}{E}{B}$ be a linear map with values in a Banach space $B$. Then $\phi$ is continuous if and only if there is $k_0\in\N$ such that $\phi$ factors through a continuous linear map $\smfunc{\widetilde \phi}{E_{k_0}}{B}$.
 \end{lem}
 \begin{proof}
  It is clear that the condition is sufficient for the continuity of $\phi$. Hence we have to establish that it is also necessary.
  Since $\smfunc{\phi}{E}{B}$ is continuous there is a continuous seminorm $p$ on $E$ such that $\norm{\phi(x)}_B\leq p(x)$ for all $x\in E$. The topology on $E=\projlim{}E_k$ is generated by the projections $\smfunc{\pi_k}{E}{E_k}$. Hence there is a $k_0\in\N$ and a constant $C>0$ such that
  \[
   \norm{\phi(x)}_B\leq p(x) \leq C \norm{\pi_{k_0}(x)}_{E_{k_0}}\quad\hbox{ for all }x\in E.
  \]
  If $x\in\ker\pi_{k_0}$, this implies at once that $\phi(x)=0$ and we obtain a well-defined map
  \[
   \func{\psi}{\pi_{k_0}(E)\subseteq E_{k_0}}{F}{\pi_{k_0}(x)}{\phi(x).}
  \]
  This linear map is continuous with operator norm at most $C$ and hence it can be extended to a continuous linear map $\tilde{\varphi} \colon E_{k_0} \rightarrow B$ as $\pi_{k_0}(E)$ is dense in $E_{k_0}$.
 \end{proof}

 
 \begin{lem}[{Lemma \ref{lem_continuity_of_linear_maps}}]				
  Let $\IndwAbs$ and $\varIndwAbs$ be graded index sets. Then a linear map $\smfunc{T}{\K^{(\Ind)}}{\K^{(\varInd)}}$ extends to a (unique) continuous operator $\widetilde T \colon \ellOneLimit{\Ind} \rightarrow \ellOneLimit{\varInd}$, if and only if for each $k_1\in\N$ there is a $k_2\in\N$ and a $C>0$ such that
  \[
   \ellOneNorm{T \tau}{k_1}\leq C \Gr{k_2}(\abs{\tau}) \quad \text{for all } \tau\in\Ind.
  \]
  \end{lem}
 \begin{proof}
  We first show that this criterion is sufficient. Note that we only need that $T$ is continuous with respect to all $\ellOneNorm{\cdot}{k}$ on the range space. 
  For $k_1\in\N$ choose $k_2$ and $C\geq0$ as in the statement of the Lemma. 
  Let $\sum_\tau c_\tau \cdot \tau\in \K^{(\Ind)}$ be given. Then
  \begin{align*}
   \ellOneNorm{T\left(	\sum_\tau c_\tau \cdot \tau	\right)}{k_1}
	      &	=	\ellOneNorm{	\sum_\tau c_\tau \cdot  T \tau	}{k_1} \leq	\sum_\tau \abs{c_\tau} \cdot  \ellOneNorm{T \tau	}{k_1}
	    \\&	\leq		\sum_\tau \abs{c_\tau} C \Gr{k_2}(\abs{\tau}) =		C \sum_\tau \abs{c_\tau} h_{k_2}(\tau)
	    =		C \ellOneNorm{\sum_\tau c_\tau \cdot \tau}{k_2}.
  \end{align*}
  Conversely, assume that the linear map $\smfunc{T}{\K^{(\Ind)}}{\K^{(\varInd)}}$ can be extended to a continuous linear operator
  $ \smfunc{\widetilde T}{\ellOneLimit{\Ind}}{\ellOneLimit{\varInd}}$.
  Then for each fixed $k_1\in\N$, we can compose this map with the projection to the Banach space $\ellOne{k_1}{\varInd}$.
  This yields a continuous linear map from $\ellOneLimit{\Ind}$ to $\ellOne{k_1}{\varInd}$.
  By Lemma \ref{lem_factorisation} this operator factors through one of the steps $\ellOne{k}{\Ind}$, i.e.~there is $k_2$ such that 
  $
   \nnfunc{\ellOne{k_2}{\Ind}}{\ellOne{k_1}{\varInd}}{\sum_\tau c_\tau \cdot \tau}{\widetilde T \left( \sum_\tau c_\tau \cdot \tau}\right) 
  $
  is a continuous operator between these Banach spaces with operator norm $0\leq C < \infty$. 
  Now each $\tau\in\Ind$ can be regarded as a vector in $\ellOne{k_2}{\Ind}$ and hence we have the estimate:
  \[
   \ellOneNorm{T \tau}{k_1} = \ellOneNorm{\widetilde T \tau }{k_1} \leq C \ellOneNorm{\tau}{k_2} = C h_{k_2}(\tau) = C \Gr{k_2}(\abs{\tau}).\qedhere
  \]
 \end{proof}
 
 
 \begin{lem}[Compact regularity of $\ellInftyLimit{\Ind}{B}$]								\label{lem_creg_lim}
  Let $\IndwAbs$ be a graded index set, $\left(\Gr{k}	\right)_{k\in\N}$ be a convex growth family and $B$ be a Banach space. 
  Then the direct limit $\ellInftyLimit{\Ind}{B}$ is \emph{compactly regular}, i.e.~for every compact set $K\subseteq \ellInftyLimit{\Ind}{B}$ there is a $k\in\N$ such that $K$ is a compact subset of $\ellInfty{k}{\Ind}{B}$.
 \end{lem}
 \begin{proof}
  By \cite[Theorem 6.1 and Theorem 6.4]{MR1977923}, compact regularity follows from 
  \begin{align*}
   \forall k_1\in\N \ \exists k_2\geq k_1 \text{ such that }\forall k_3\geq k_2 \colon (\forall \epsilon>0 \exists \delta>0)\colon \\ \oBallin{1}{\ellInfty{k_1}{\Ind}{B}}{0} \cap \oBallin{\delta}{\ellInfty{k_3}{\Ind}{B}}{0} \subseteq  \oBallin{\epsilon}{\ellInfty{k_2}{\Ind}{B}}{0}.
  \end{align*}
  For fixed $k_1\in\N$, we need $k_2$ such that for all $k_3 \geq k_2 $, $\epsilon>0$ there exists $\delta>0$ with
  \begin{displaymath}
   \sup_{\tau} \frac{\abs{f(\tau)}}{\Gr{k_1} (|\tau|)} \leq 1 \text{ and } \sup_{\tau}\frac{\abs{f (\tau)}}{\Gr{k_3}(|\tau|)} \leq \delta \text{ implies }  \sup_{\tau} \frac{\abs{f(\tau)}}{\Gr{k_2} (|\tau|)} \leq \epsilon.
  \end{displaymath}
  Let $k_1\in\N$ be given. We choose $k_2\geq k_1$ as in \ref{axiom_H_convexity} and fix some $\varepsilon >0$. 
  We will now construct a suitable $\delta >0$ independent of $\tau \in \Ind$ which satisfies the above property. 
  By \ref{axiom_H_monotonic_in_k} we know that $\Gr{k_1} (|\tau|) < \Gr{k_2}(|\tau|)$ for all $\tau$. Hence for $\alpha \in ]0,1[$ we obtain the estimate 
  $
   \Gr{k_1} (\cdot)^\alpha \cdot \Gr{k_3}(\cdot)^{1-\alpha} \leq \Gr{k_2} (\cdot).
$
  Thus 
  \begin{align*}
   |f(\tau)| &= \left(\frac{|f(\tau)|}{\Gr{k_1}(|\tau|)} \right)^\alpha \left(\frac{|f(\tau)|}{\Gr{k_3} (|\tau|)} \right)^{1-\alpha} \Gr{k_1}(|\tau|)^\alpha \cdot \Gr{k_3}(|\tau|)^{1-\alpha} \\
            &\leq \left(\frac{|f(\tau)|}{\Gr{k_1}(|\tau|)} \right)^\alpha \left(\frac{|f(\tau)|}{\Gr{k_3}(|\tau|)} \right)^{1-\alpha}  \Gr{k_2}(|\tau|)   \leq \delta^{1-\alpha} \Gr{k_2} (|\tau|).
  \end{align*}
  This is independent of $\tau$, whence the property holds for all $\delta>0$ with $\delta^{1-\alpha} < \varepsilon$. 
 \end{proof}


 \begin{lem}[$\ellInftyLimit{\Ind}{B}$ is sometimes a Silva space]											\label{lem_silva}
  Let $\left(\Gr{k}\right)_{k\in\N}$ be a (possibly not convex) growth family and assume that $\IndwAbs$ is a graded index set of finite type. 
  Assume further that $\text{dim }B<\infty$.
  Then the \textup{(LB)}-space $\ellInftyLimit{\Ind}{B}$ is a \emph{Silva space} and \emph{compactly regular}, i.e.~a locally convex direct limit with compact operators as bonding maps such that every compact subset of $\ellInftyLimit{\Ind}{B}$ is a compact subset of some step.
 \end{lem}
 \begin{proof}
  Every Silva space is compactly regular (see e.g.~\cite[Proposition 4.4]{MR2743766}). Thus it suffices to show that $\ellInftyLimit{\Ind}{B}$ is a Silva space.
  To this end, we construct iteratively indices which satisfy \ref{axiom_H_Infinity}.
  This yields a sequence $k_1<k_2<k_3<k_4<\cdots$ such that
  \[
   \frac{ \Gr{k_j}(n)}{\Gr{k_{j+1}}(n)}\leq \frac{1}{2^n}		\quad \hbox{ for all } n\in\N
  \]
  Since passing to a cofinal subsequence $\left(\ellInfty{k_1}{\Ind}{B} \to \ellInfty{k_2}{\Ind}{B} \to \cdots \right)$ does not change the limit, we may assume without loss of generality that $k_j = j$ and thus that
  \begin{equation}													\label{eq silva}
   \frac{ \Gr{k}(n)}{\Gr{k+1}(n)}\leq \frac{1}{2^n}			\quad \hbox{ for all } n\in\N			
  \end{equation}
  Using the isometric isomorphisms 
  $\ellInfty{k}{\Ind}{B} \rightarrow \ellInfty{}{\Ind}{B}, f\mapsto \frac{f}{\Gr{k}(\abs{\cdot})}$, we see that the bonding map $I_{k}^{k+1}$ is a compact operator if and only if the multiplication operator
 $
   M_h \colon \ellInfty{}{\Ind}{B} \rightarrow \ellInfty{}{\Ind}{B}, f \mapsto h\cdot f,
  $
 where $\func{h}{\Ind}{\R}{\tau}{\frac{ \Gr{k}(\abs{\tau})}{\Gr{k+1}(\abs{\tau})}}$, is compact.

 Since $B\cong\K^d$, we identify $\ellInfty{}{\Ind}{\K^d}\cong \left( \ellInfty{}{\Ind}{\K}	\right)^d,$ so without loss of generality, $B=\K$.
 In this case, the multiplication operator 
 is compact if and only if $h$ 
 has countable support and $\lim_{j\to\infty} \Gr{k}/\Gr{k+1}(|\tau_j|) = 0$ for some enumeration $(\tau_j)_j$ of the support of $h$ (cf.\ \cite[p.\ 147 Exercise 4.]{MR1483073}). 
 Since $\IndwAbs$ is of finite type, this means that $\abs{\tau_j} \rightarrow \infty$, whence $\Gr{k}/\Gr{k+1}(|\tau_j|) \leq 2^{-|\tau_j|} \rightarrow 0$ by \eqref{eq silva}. Thus the bonding maps are compact and $\ellInftyLimit{\Ind}{B}$ is a Silva space.
 \end{proof} 
 
 \phantomsection
\addcontentsline{toc}{section}{References}
\bibliographystyle{new}
\bibliography{controlled_lit} 
\end{document}